\documentclass[11pt]{article} 

\usepackage{amsmath}
\usepackage{lmodern}
\usepackage[T1]{fontenc}
\usepackage[babel=true]{microtype}
\usepackage{amsxtra}%
\usepackage{amsfonts}%
\usepackage{amssymb}%
\usepackage{amsthm}
\usepackage{graphicx}
\usepackage{epstopdf}
\usepackage{color,hyperref}
\hypersetup{colorlinks,breaklinks,
             linkcolor=blue,urlcolor=blue,
           anchorcolor=blue,citecolor=blue}
       
\usepackage{subfigure}
\usepackage{appendix}

\usepackage[margin=2.54cm]{geometry}

\newcommand{\eps}{\varepsilon}

\newcommand{\R}{\mathbb{R}}

\newcommand{\C}{\mathbb{C}}

\renewcommand{\phi}{\varphi}

\newcommand{\mcl}{\mathcal{L}}

\usepackage{enumitem}
\setitemize{noitemsep,topsep=0pt,parsep=2pt,partopsep=0pt}
\renewcommand{\Re}{\mathrm{Re} \,}
\renewcommand{\Im}{\mathrm{Im} \,}

\def\XXint#1#2#3{{\setbox0=\hbox{$#1{#2#3}{\int}$ }
		\vcenter{\hbox{$#2#3$ }}\kern-.6\wd0}}

\newtheorem{thm}{Theorem}

\newtheorem*{thm*}{Theorem}

\newtheorem{prop}{Proposition}
\newtheorem{lemma}[prop]{Lemma}
\newtheorem{corollary}[prop]{Corollary}

\theoremstyle{definition}

\numberwithin{equation}{section}
\numberwithin{prop}{section}

\begin{document}
\begin{center}
{\fontsize{15}{15}\fontseries{b}\selectfont{Sharp decay rates for localized perturbations to the critical front in the Ginzburg-Landau equation}}\\[0.2in]
Montie Avery and Arnd Scheel \\[0.1in]
\textit{\footnotesize 
University of Minnesota, School of Mathematics,   206 Church St. S.E., Minneapolis, MN 55455, USA} 

\textit{\footnotesize Dedicated to the memory of Pavol Brunovsk\'y}
\end{center}

\begin{abstract}
	We revisit the nonlinear stability of the critical invasion front in the Ginzburg-Landau equation. Our  main result shows that the amplitude of localized perturbations decays with rate $t^{-3/2}$, while the phase decays diffusively. We thereby refine earlier work of Bricmont and Kupiainen as well as Eckmann and Wayne, who separately established nonlinear stability but with slower decay rates. On a technical level, we rely on sharp linear estimates obtained through analysis of the resolvent near the essential spectrum via a far-field/core decomposition which is well suited to accurately describing the dynamics of separate neutrally stable modes arising from far-field behavior on the left and right. 
\end{abstract}

\section{Introduction}

The Ginzburg-Landau equation
\begin{align}
	A_t = A_{xx} + A - A |A|^2, \quad A = A(x,t) \in \C. \label{e: GL}
\end{align}
arises in many contexts as a modulation equation describing approximate dynamics near a Turing instability in pattern-forming systems. In many cases, patterns nucleate locally near a small, localized perturbation of the trivial background state, and then grow, saturate, and spread through a spatially extended system. A fundamental question is then to determine the speed at which localized disturbances spread through the system and the pattern this invasion mechanism produces in the wake. One often restricts mathematical considerations to a one-sided invasion processes, in which case a first description of the invasion process focuses on the existence and stability of \textit{invasion fronts} connecting the stable and unstable rest states. In the Ginzburg-Landau equation, these are traveling wave solutions $A(x,t) = q(x-ct;c)$ satisfying
\begin{align}
	q'' + c q' + q - q^3 = 0, \quad \lim_{\xi \to -\infty} q (\xi; c) = 0, \quad \lim_{\xi \to \infty} q (\xi; c) = 1. 
\end{align}
For each fixed $c$, $q(\cdot; c)$ generates a two-parameter family of invasion fronts with speed $c$ via spatial translation and rotation of the complex phase, owing to the translation invariance and gauge symmetry of \eqref{e: GL}. The real fronts are monotone for $c \geq 2$, and the front with the minimal speed $c_* = 2$ is the most interesting in light of the \textit{marginal stability conjecture} \cite{vanSaarloosReview, Bricmont, EckmannWayne}, which postulates that solutions to \eqref{e: GL} with compactly supported initial data, the most relevant in most invasion processes, spread with asymptotic speed 2. We therefore write $q_* = q(\cdot;2)$, and refer to this solution as the \textit{critical front}. 

When restricted to real-valued solutions, \eqref{e: GL} obeys a maximum principle, and one can then use comparison principle based arguments to show that non-negative, compactly supported initial data spread with asymptotic speed 2 \cite{aronson, Bramson1, Bramson2, Comparison1, Comparison2, Comparison3}. The lack of a maximum principle for complex-valued solutions, however, presents a substantial challenge to resolving the marginal stability conjecture in \eqref{e: GL}. The present authors recently proved the marginal stability conjecture in a general framework of higher order parabolic equations, which in particular lack maximum principles, under conceptual assumptions on the existence and spectral stability of critical fronts \cite{AveryScheelSelection}. In a broader setting, the minimal speed $c_* = 2$ is replaced by the \textit{linear spreading speed}, which characterizes marginal pointwise linear stability in the co-moving frame; see \cite{AveryScheelSelection, HolzerScheelPointwiseGrowth} for details. The analysis in \cite{AveryScheelSelection}, however, relies on an additional technical assumption that the invading state in the wake of the fronts is exponentially stable, with spectrum strictly contained in the left half plane. This assumption is violated here due to the gauge invariance of \eqref{e: GL}, with the invading state instead being only diffusively stable; see Figure \ref{f: spectrum and schematic}. 


Nonlinear stability of the critical front in the Ginzburg-Landau equation against sufficiently localized perturbations was established in \cite{EckmannWayne, Bricmont}. The analysis in \cite{EckmannWayne} is based on energy estimates, and establishes stability without identifying a precise decay rate, while \cite{Bricmont} gives a more detailed description of the dynamics via renormalization group theory, establishing stability with decay rate $t^{-1+\eps}$ in the amplitude and $t^{-1/2+\eps}$ in the phase. Here we revisit this stability analysis and obtain sharp decay rates for the amplitude and phase of perturbations, thereby improving upon the results of \cite{Bricmont, EckmannWayne}. We consider \eqref{e: GL} in the co-moving frame with speed $c_* = 2$, so that $q_*$ is an equilibrium solution to the resulting equation
\begin{align}
	A_t = A_{xx} + 2 A_x + A - A |A|^2, \label{e: GL comoving}
\end{align}

and consider complex-valued perturbations of $q_*$ of the form $A = (q_* + r)e^{i\phi}$. To state our main result, we first introduce a smooth positive exponential weight $\omega$ satisfying 
\begin{align}
	\omega(x) = \begin{cases}
		e^{x}, & x \geq 1, \\
		1, & x \leq -1,
	\end{cases} \label{e: omega def}
\end{align}
as well as smooth positive algebraic weights $\rho_{r_-, r_+}$ for $r_-, r_+ \in \R$, satisfying 
\begin{align}
	\rho_{r_-, r_+} (x) = \begin{cases}
		(1+x^2)^{r_+/2}, & x \geq 1, \\
		(1+x^2)^{r_-/2}, & x \leq -1. 
	\end{cases}
	\label{e: rho def}
\end{align}
Our main result is the following nonlinear stability of the critical front. 
\begin{thm}\label{t: main}
	There exist positive constants $C$ and $\eps$ so that if $(r_0, \phi_0) \in L^1 \cap L^\infty(\R) \times L^1 \cap W^{1,\infty} (\R)$ satisfy
	\begin{align}
		\| \omega \rho_{0,1} r_0 \|_{L^1} + \| \omega r_0 \|_{L^\infty} + \|  \phi_0 \|_{L^1\cap W^{1,\infty}} < \eps, 
	\end{align}
	then the solution $A = (q_* + r)e^{i \phi}$ to \eqref{e: GL comoving} with initial data $A_0 = (q_* + r_0)e^{i \phi_0}$ exists for all $t > 0$ and satisfies 
	\begin{align}
		\| \rho_{0, -1} \omega r(\cdot, t) \|_{W^{1,\infty}} &\leq \frac{C \eps}{(1+t)^{3/2}}, \\
		\| \rho_{0, -1} \phi(\cdot, t) \|_{W^{1, \infty}} &\leq \frac{C \eps}{(1+t)^{1/2}}
	\end{align}
	for all $t > 0$. 
\end{thm}
As usual, $W^{k, \ell} (\R)$ denotes the Sobolev space of weakly differentiable functions up to order $k$ with integrability index $\ell$. Due to the gauge invariance of \eqref{e: GL}, Theorem \ref{t: main} of course holds if one replaces $q_*$ with $q_* e^{i \theta_0}$ for any fixed phase $\theta_0 \in [0, 2\pi)$. The $t^{-3/2}$ decay exhibited here for the amplitude was established under the restriction to real-valued solutions in \cite{Gallay, FayeHolzer, AveryScheel} and is known to be sharp in light of the asymptotics given in \cite{Gallay, AveryScheel}. Key to the analysis of \cite{AveryScheelSelection} establishing the marginal stability conjecture for an exponentially stable invading state are sharp decay estimates for perturbations of the critical front, used there to close a perturbative argument near a refined approximate solution. Indeed, the $t^{-3/2}$ decay rate is closely related to the logarithmic delay $-\frac{3}{2} t$ in the position of a front solution evolving from initial data compactly supported on the right as predicted by Ebert and van Saarloos \cite{EbertvanSaarloos}. In light of this, we are confident that the present analysis is not only of technical interest but also represents a significant step towards resolving the marginal stability conjecture in the Ginzburg-Landau equation. 

\subsection{Preliminaries}

\noindent \textbf{Choice of coordinates and general approach.} Considering perturbations of the critical front of the form $A = (q_* + r) e^{i \phi}$ leads to the system 
\begin{align}
	r_t &= r_{xx} + 2 r_x + (1- 3 q_*^2)r - 3 q_* r^2 - r^3 - q_* \phi_x^2 - r \phi_x^2, \\
	\phi_t &= \phi_{xx} + 2 \phi_x + 2 \frac{\partial_x (q_* + r)}{q_* + r} \phi_x. 
\end{align} 
The linearization about $(r, \phi) = (0, 0)$ is diagonal, and Palmer's theorem \cite{Palmer1, Palmer2} implies that the essential spectrum of the $r$ component of the linearization is unstable, since $1- 3 q_*^2 \to 1$ as $x \to \infty$.  Conjugating with the exponential weight $\omega$ defined in \eqref{e: omega def} stabilizes the essential spectrum, so that it touches the imaginary axis at the origin and is otherwise contained in the left half-plane; see Figure \ref{f: spectrum and schematic}, left panel. Hence we define the weighted variable $p = \omega r$, so that we recover this marginal stability by restricting to exponentially localized perturbations. 

\begin{figure}
	\centering
	\includegraphics[width = 1\textwidth]{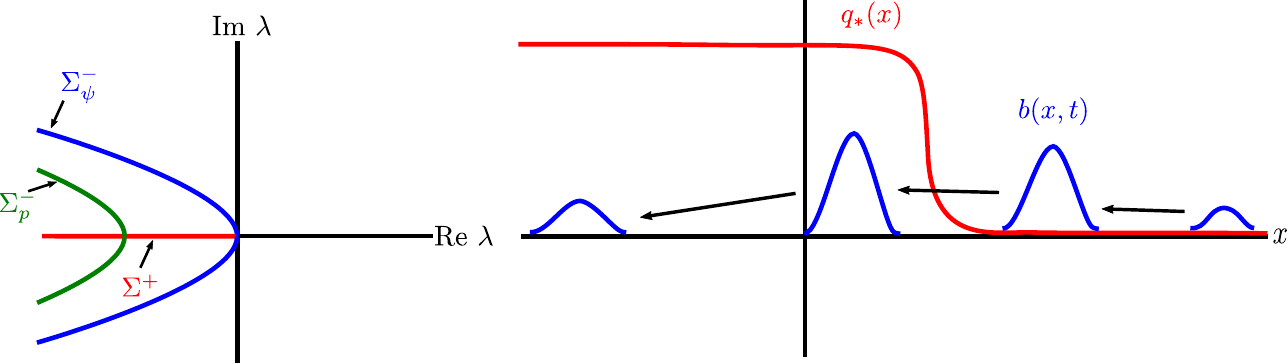}
	\caption{\footnotesize{Left: the Fredholm borders of the linearizations $\mcl_\psi$ and $\mcl_p$, which bound their respective essential spectra. Right: a schematic of a real critical front (red) and the time evolution of an exponentially localized perturbation in the imaginary component $b(x,t)$, in the co-moving frame with speed 2. The perturbation initially grows due to the pointwise instability of the background state, but is advected into the bulk of the front where it decays diffusively. This imaginary perturbation to the real front also induces a faster-decaying real perturbation to the real front, which we omit in this picture.}}
	\label{f: spectrum and schematic}
\end{figure}

The linearization in the $\phi$ component is $\mcl_\phi = \partial_{xx} + 2 \partial_x + 2 q_*'/q_* \partial_x$. The essential spectrum of this operator is marginally stable but we encounter an additional technical difficulty, namely that the coefficient $2q_*'/q_*$ attains its limit at $+\infty$ at only an algebraic rate. As we shall see, this slow algebraic convergence would obstruct our approach to obtaining linear estimates, and we therefore remove this difficulty by introducing a weighted variable $\psi = \omega q_* \phi$. These coordinates are used in a heuristic argument for the expected decay rates in \cite{Bricmont}, and are similar to those used to establish nonlinear stability of source defects in the complex-coefficient Ginzburg-Landau equation in \cite{BeckNguyenSandstedeZumbrunCGL}.

An alternative approach would separate $A = a + ib$ into real and imaginary components; indeed, these coordinates are natural in that the rest state $A \equiv 0$ admits a two-dimensional linear pointwise instability, with separate growth in the real and imaginary components. Hence the linearization about a critical front in the imaginary component is also unstable, requiring an exponential weight to push the essential spectrum to the imaginary axis. This is somehow masked in $(r, \phi)$ coordinates due to the singularity in polar coordinates as $r \to 0^+$. The technical difficulty with using $(a,b)$ coordinates is that in these coordinates, critical nonlinear terms appear in the analysis of the state in the wake, so that a normal form-type coordinate transformation would be needed to remove these terms and close a nonlinear argument. We note here that $\psi(x, t) \sim e^x b(x,t)$ for $x \gg 1$, so that $\psi$ captures the behavior of the imaginary component in the leading edge, while simultaneously enjoying the advantage that polar coordinates capture only irrelevant nonlinear terms in the wake.

The weighted variables $(p, \psi)$ then solve the system
\begin{align}
	p_t &= \mcl_p p - 3 q_* \omega^{-1} p^2 - \omega^{-2} p^3 - \omega q_* [ (\omega^{-1} q_*^{-1}  \psi)_x]^2 - p [(\omega^{-1} q_*^{-1} \psi)_x]^2 \label{e: p equation}\\
	\psi_t &= \mcl_\psi \psi + 2 \omega q_*' \left( \frac{1}{1+\omega^{-1} q_*^{-1} p} - 1 \right) (\omega^{-1} q_*^{-1} \psi)_x + 2 \omega (\omega^{-1} p)_x \left( \frac{1}{1 + \omega^{-1} q_*^{-1}} \right) (\omega^{-1} q_*^{-1} \psi)_x, \label{e: psi equation}
\end{align}
where 
\begin{align}
	\mcl_p = \partial_{xx} + 2(1 + \omega (\omega^{-1})') \partial_x + 1 - 3 q_*^3 + 2 \omega (\omega^{-1})' + \omega \omega''. 
\end{align}
and 
\begin{align}
	\mcl_\psi = \partial_{xx} + 2 (1 + (\omega^{-1})' \omega) \partial_x + (\omega^{-1})'' \omega + 1 - q_*^2. 
\end{align}
The coefficients of $\mcl_\psi$ and $\mcl_p$ each attain limits exponentially quickly as $x \to \pm \infty$, with limiting operators 
\begin{align}
	\mcl_p^+ = \partial_{xx}, \quad \mcl_p^- = \partial_{xx} + 2 \partial_x - 2,
\end{align}
and 
\begin{align}
	\mcl_\psi^+ = \partial_{xx}, \quad \mcl_\psi^- = \partial_{xx} + 2 \partial_x. 
\end{align}
The essential spectra of these operators are given by the dispersion curves 
\begin{align}
	\sigma_{\mathrm{ess}} (\partial_{xx}) &:= \Sigma^+ = \{ -k^2 \in \C : k \in \R \}, \\
	\sigma_{\mathrm{ess}} (\mcl_p^-) &:= \Sigma_p^- = \{ -k^2 + 2 ik - 2 \in \C : k \in \R\}, \\
	\sigma_{\mathrm{ess}} (\mcl_\psi^-) &:= \Sigma_\psi^- = \{ -k^2 + 2 ik \in \C : k \in \R\}. 
\end{align}
These curves determine the boundaries of the essential spectra of the full operators $\mcl_p$ and $\mcl_\psi$; see Figure \ref{f: spectrum and schematic}, left panel, for a schematic and \cite{KapitulaPromislow, FiedlerScheel} for background. Our analysis is based on sharp linear decay estimates which we obtain by deforming the integration contours in the definitions of the semigroups $e^{\mcl_p t}, e^{\mcl_\psi t}$ near the essential spectrum. We then extract time decay from precise estimates on the behavior of the resolvents of $\mcl_p, \mcl_\psi$ near their essential spectrum. Crucial to our approach is the use of a far-field/core decomposition which allow us to efficiently separate behavior arising from the limiting dynamics at $+\infty$ from that determined by dynamics at $-\infty$. We are thereby able to decompose the resolvent into two terms and deform the integration contours in the formula for the semigroup separately in each of these terms, using contours adapted to the behavior on the right as in \cite{AveryScheel, AveryScheelSelection} in one case and contours adapted to the diffusive spectrum with quadratic tangency at the origin, $\Sigma_\psi^-$, as in \cite{Kapitula1, Kapitula2}, in the other. 

The stability of the critical front considered here bears some conceptual similarities to the stability of source defects in the complex Ginzburg-Landau equation considered in \cite{BeckNguyenSandstedeZumbrunCGL}. In particular, a key challenge in both contexts is characterizing diffusive stability in the presence of outward transport. The approach in \cite{BeckNguyenSandstedeZumbrunCGL} uses an ansatz to explicitly capture outgoing diffusive wave packets and then establish decay using pointwise semigroup estimates. Here, rather than explicitly capturing the diffusive wave packets advected to the left in the bulk of the front (see Figure \ref{f: spectrum and schematic}), we take advantage of the fact that the outward transport induces additional decay in weighted norms which allow for algebraic growth. We are able to estimate the nonlinearities in such norms due to the fact that only derivatives of $\phi$ appear in the $(r, \phi)$ system, and hence when we change to $(r, \psi)$ coordinates, every term in the nonlinearity term involving $\psi$, rather than $\psi_x$, carries a factor of $(\omega^{-1} q_*^{-1})_x$ and hence is very localized on the left. This is ultimately due to the gauge invariance in the original coordinates in \eqref{e: GL}. The improved decay of $\psi$ in algebraically weighted norms allowing growth as well as improved decay of $\psi_x$ when compared to the diffusive decay rate $t^{-1/2}$ is then sufficient to close a nonlinear argument. Our approach to the nonlinear argument is therefore somewhat more direct than that of \cite{BeckNguyenSandstedeZumbrunCGL}, but at the cost of a detailed description of the outgoing wave packets. 

We also mention related work establishing stability of supercritical fronts -- moving faster than the linear spreading speed -- in the Swift-Hohenberg equation \cite{SchneiderEckmann} and in the Ginzburg-Landau equation coupled to an additional conservation law \cite{Hilder}, where the main difficulty is again to characterize diffusive decay in the presence of outward transport. The methods there are specifically adapted to supercritical fronts, relying crucially on the fact that one can obtain exponential in time linear stability of the unstable rest state in a suitable exponentially weighted norm. 

\noindent \textbf{Outline of the paper.} In Section \ref{s: p resolvent estimates}, we revisit the study of the resolvent of $\mcl_p$ in order to obtain new decay estimates for $e^{\mcl_p t}$ in $L^1(\R)$, as we will need these estimates to close our nonlinear arguments. In Section \ref{s: psi resolvent estimates}, we study the resolvent of $\mcl_{\psi}$, separating behavior originating from limiting dynamics on the left and on the right via a far-field core decomposition, extending techniques introduced in \cite{AveryScheel, AveryScheelSelection}. In Section \ref{s: linear estimates}, we translate the resolvent estimates obtained in the preceding two sections into linear decay estimates via integrating over appropriately chosen contours. In Section \ref{s: nonlinear stability}, we use control of a carefully constructed time-weighted norm to show that the linear decay estimates persist for the full nonlinear system, thereby proving Theorem \ref{t: main}.

\noindent \textbf{Function spaces.} We require more general exponential weights for our analysis of the resolvent near the essential spectrum. For $\eta_-, \eta_+ \in \R$, we define a smooth positive exponential weight $\omega_{\eta_-, \eta_+}$ satisfying 
\begin{align}
	\omega_{\eta_-, \eta_+} (x) = \begin{cases}
		e^{\eta_+ x}, & x \geq 1, \\
		e^{\eta_- x}, & x \leq -1.  
	\end{cases}
\end{align}
For an non-negative integer $k$ and $1 \leq \ell \leq \infty$, we then define the exponentially weighted Sobolev space $W^{k,\ell}_{\mathrm{exp},\eta_-,\eta_+} (\R)$ through the norm 
\begin{align}
	\| f \|_{W^{k,\ell}_{\mathrm{exp},\eta_-,\eta_+}} = \| \omega_{\eta_-, \eta_+} f\|_{W^{k,\ell}}. 
\end{align}
When $k = 0$, we write $W^{0, \ell}_{\mathrm{exp}, \eta_-, \eta_+} (\R) = L^\ell_{\mathrm{exp},\eta_-, \eta_+} (\R)$. Similarly, for $r_-, r_+ \in \R$, we define algebraically weighted Sobolev spaces $W^{k, \ell}_{r_-, r_+} (\R)$ through the norm 
\begin{align}
	\| f \|_{W^{k,\ell}_{r_-, r_+}} = \| \rho_{r_-, r_+} f \|_{W^{k, \ell}},
\end{align}
where $\rho_{r_-, r_+}$ is given by \eqref{e: rho def}, and for $k = 0$ we write $W^{0, \ell}_{r_-, r_+} (\R) = L^\ell_{r_-, r_+} (\R)$. 

\noindent \textbf{Additional notation.} We let $\mathcal{B}(X,Y)$ denote the space of bounded linear operators between two Banach spaces $X$ and $Y$, equipped with the operator norm topology. For $\delta > 0$, we let $B(0, \delta)$ denote the open unit ball centered at the origin in the complex plane with radius $\delta$. When the intention is clear, we may abuse notation slightly by writing a function $u(x,t)$ or $u(x;\gamma)$ as $u(t) = u(\cdot, t)$ or $u(\gamma) = u(\cdot; \gamma)$, viewing it as an element of some function space for each $t$ or $\gamma$. Throughout the paper, we use the notation $\langle x \rangle = (1+x^2)^{1/2}$.

\noindent \textbf{Acknowledgements.} This material is based upon work supported by the National Science
Foundation through the Graduate Research Fellowship Program under Grant No. 00074041, as
well as through NSF-DMS-1907391. Any opinions, findings, and conclusions or recommendations
expressed in this material are those of the authors and do not necessarily reflect the views of the
National Science Foundation.

\section{Resolvent estimates for $\mcl_p$}\label{s: p resolvent estimates}

The linearization in the amplitude component $\mcl_p$ is precisely the linearization about a real Fisher-KPP front, which we have studied in greater generality via our far-field/core approach in \cite{AveryScheel, AveryScheelSelection}. To unfold the branch point in the right dispersion curve $\Sigma^+$, we let $\gamma = \sqrt{\lambda}$, with branch cut chosen along the negative real axis. The sharp $t^{-3/2}$ decay rate is implied by the following regularity of the resolvent near the essential spectrum, which is a special case of \cite[Proposition 3.5]{AveryScheelSelection}. 

\begin{prop}\label{p: p resolvent lipschitz}
	There exist positive constants $C$ and $\delta$ and a bounded limiting operator $R_p^0 : L^1_{0,1} (\R) \to W^{1, \infty}_{0, -1} (\R)$ such that for any $f \in L^1_{0,1} (\R)$, we have 
	\begin{align}
		\| (\mcl_p - \gamma^2)^{-1} f - R_p^0 f \|_{W^{1, \infty}_{0, -1}} \leq C |\gamma| \| f \|_{L^1_{0,1}}
	\end{align}
	for all $\gamma \in B(0,\delta)$ such that $\Re \gamma \geq 0$. 
\end{prop}

To close the nonlinear argument here, we require an additional linear decay estimate measuring the solution in $L^1$, which we prove by estimating the resolvent near the essential spectrum in $L^1$. We start by analyzing the resolvent of the limiting operator on the right, $\mcl_p^+ = \partial_{xx}$. 

\subsection{Resolvent estimates for $\mcl_p^+$}
As in \cite{AveryScheel, AveryScheelSelection}, we take advantage of the absorption mechanism induced by the strong spectral stability of $\mcl_p^-$ by establishing estimates on $(\mcl_p^+ - \gamma^2)^{-1}$ restricted to odd functions, and then enforcing this oddness in our far-field/core decomposition when we pass estimates to the full resolvent. For any sufficiently localized odd function $f$, the action of $(\mcl_p^+-\gamma^2)^{-1} = (\partial_{xx}-\gamma^2)^{-1}$ for any $\gamma \in \C$ with $\Re \gamma > 0$ is given by 
\begin{align}
	(\partial_{xx} - \gamma^2)^{-1} f (x) = \int_0^\infty G_{\gamma}^\mathrm{odd} (x,y) f(y) \, dy,
\end{align} 
with integral kernel 
\begin{align}
	G_\gamma^\mathrm{odd} (x,y) = \frac{1}{2 \gamma} \left( e^{-\gamma |x-y|} - e^{-\gamma|x+y|} \right).
\end{align}
Using this representation, we establish the following estimate. 

\begin{lemma}\label{l: p resolvent right L11 to L1}
	There exists a constant $C > 0$ such that for any odd $f \in L^1_{1,1} (\R)$, we have 
	\begin{align}
		\| (\partial_{xx}-\gamma^2)^{-1} f \|_{L^1} \leq \frac{C}{|\gamma|} \| f \|_{L^1_{1,1}}
	\end{align}
	for any $\gamma$ with $\Re \gamma \ge \frac{1}{2} | \Im \gamma|$. 
\end{lemma}
\begin{proof}
	First we establish the pointwise estimate 
	\begin{align}
		| G_\gamma^\mathrm{odd} (x,y) | \leq C e^{-c|\gamma| |x-y|} \langle y \rangle \label{e: right resolvent G odd pointwise estimate}
	\end{align}
	for $x, y \geq 0$, for some constants $C, c > 0$, and for any $\gamma$ with $\Re \gamma \ge \frac{1}{2} | \Im \gamma|$. To prove this estimate, first consider the case $x \geq y \geq 0$, for which we have 
	\begin{align*}
		G_\gamma^\mathrm{odd} (x,y) = \frac{1}{2 \gamma} e^{-\gamma (x-y)} \left( 1 - e^{-2 \gamma y} \right). 
	\end{align*}
	The restriction $\Re \gamma \geq \frac{1}{2} | \Im \gamma|$ implies that $-\Re \gamma \leq -c |\gamma|$ for some constant $c > 0$. Together with the fact that $| 1 - e^z | \leq C |z|$ for $\Re z \leq 0$, we thereby obtain 
	\begin{align*}
		|G_\gamma^\mathrm{odd}(x,y) | \leq \frac{C}{|\gamma|} e^{-c|\gamma| (x-y)} (2 |\gamma| |y|) \leq C e^{-c|\gamma| |x-y|} \langle y \rangle
	\end{align*}
	for $x \geq y \geq 0$. For $y \geq x \geq 0$, the same argument with the roles of $x$ and $y$ interchanged leads to the estimate 
	\begin{align*}
		|G_\gamma^\mathrm{odd} (x,y)| \leq C e^{-c|\gamma| (y-x)} \langle x \rangle \leq C e^{-c|\gamma| |x-y|}  \langle y \rangle,
	\end{align*}
	since $\langle x \rangle \leq \langle y \rangle$ for $0 \leq x \leq y$, and hence we have the desired pointwise estimate. Using this estimate, we obtain 
	\begin{align*}
		\| (\partial_{xx} - \gamma^2)^{-1} f \|_{L^1} &\leq C \int_\R \left| \int_0^\infty e^{-c |\gamma| |x-y|} \langle y \rangle | f(y) | \, dy \right| \, dx \\
		&\leq C \int_\R \langle y \rangle | f(y) | \, \int_\R e^{-c |\gamma| |x-y|} \, dx \, dy \\
		&\leq \frac{C}{|\gamma|} \| f \|_{L^1_{1,1}}, 
	\end{align*}
	as desired. 
\end{proof}

The following estimate establishes boundedness of the resolvent in $L^\infty$ provided an extra factor of exponential localization, and is useful in passing resolvent estimates onto the core terms in our far-field/core decomposition. 
\begin{lemma}\label{l: right resolvent exponentially localized boundedness}
	Fix $\eta > 0$. There exist positive constants $C$ and $\delta$ such that for any odd $f \in L^1(\R)$, we have 
	\begin{align}
		\sup_{x \geq 0} e^{-\eta x} | (\partial_{xx}-\gamma^2)^{-1} f (x) | \leq C \| f \|_{L^1} 
	\end{align}
	for all $\gamma \in B(0, \delta)$ with $\Re \gamma \geq 0$. 
\end{lemma}
\begin{proof}
	The result follows from the pointwise estimate 
	\begin{align}
		e^{-\eta x}| G_\gamma^\mathrm{odd} (x,y) | \leq C \label{e: p right resolvent exponentially localized estimate}
	\end{align}
	for $x, y \geq 0$, and $\Re \gamma \geq 0$. To prove this, first consider $x \geq y \geq 0$. In this case, arguing as the preceding lemma, we have 
	\begin{align*}
		e^{-\eta x} | G_\gamma^\mathrm{odd} (x,y) | \leq C e^{-\eta x} \langle x \rangle | e^{-\gamma (x-y)}| = C e^{-\eta x} \langle x \rangle | e^{-\gamma x} | | e^{\gamma y}|. 
	\end{align*}
	Since $x \geq y$, we have $e^{-\eta x} \leq e^{-\frac{\eta}{2} x} e^{-\frac{\eta}{2} y}$, and hence if $\gamma$ is sufficiently small relative to $\eta$, $e^{-\eta x} \langle x \rangle | e^{-\gamma x} | | e^{\gamma y}|$ is bounded, so that \eqref{e: p right resolvent exponentially localized estimate} holds for $x \geq y \geq 0$. For $y \geq x \geq 0$, we again argue as in the preceding lemma to instead obtain
	\begin{align*}
			e^{-\eta x}| G_\gamma^\mathrm{odd} (x,y) | \leq C e^{-\eta x} \langle x \rangle | e^{-\gamma (y-x)}| = C e^{-\eta x} \langle x \rangle | e^{-\gamma y}| | e^{\gamma x}|\leq C, 
	\end{align*} 
	as desired. 
\end{proof}

Finally, we state a basic estimate which corresponds to the standard $L^1$-$L^\infty,$ $t^{-1/2}$ decay estimate in the heat equation, which will prove useful in establishing the same estimate for $e^{\mcl_\psi t}$. This estimate follows readily from Young's convolution inequality. 
\begin{lemma}\label{l: right resolvent L1 to Linf estimate}
	There exists a constant $C > 0$ such that for any $f \in L^1 (\R)$, we have 
	\begin{align}
		\| (\partial_{xx} - \gamma^2)^{-1} f \|_{L^\infty} \leq \frac{C}{|\gamma|} \| f \|_{L^1} 
	\end{align}
	for any $\gamma$ with $\Re \gamma \geq \frac{1}{2} | \Im \gamma |$. 
\end{lemma}

\subsection{Full resolvent estimates for $\mcl_p$}\label{s: p full resolvent estimates}
In order to establish the equivalent of Lemma \ref{l: p resolvent right L11 to L1} for the full resolvent $(\mcl_p - \gamma^2)^{-1}$, we revisit the far-field/core decomposition used to prove Proposition \ref{p: p resolvent lipschitz} in greater generality in \cite{AveryScheel, AveryScheelSelection}. We first let $(\chi_-, \chi_c, \chi_+)$ be a partition of unity on $\R$ with
\begin{align*}
	\chi_+ (x) = \begin{cases}
		0, & x \leq 2, \\
		1, & x \geq 3,
	\end{cases}
\end{align*}
and $\chi_- (x) = \chi_+ (-x)$. Hence $\chi_c (x)$ is compactly supported. We decompose a given $f \in L^1_{0,1} (\R)$ as 
\begin{align}
	f = \chi_- f + \chi_c f + \chi_+ f =: f_- + f_c + f_+. 
\end{align}
We let $p^+$ solve 
\begin{align}
	(\mcl_p^+ - \gamma^2) p^+ = f_+^\mathrm{odd},
\end{align}
where $f_+^\mathrm{odd} = f_+ (x) - f_+ (-x)$ is the odd extension of $f_+$. We let $p^-$ solve 
\begin{align}
	(\mcl_p^- - \gamma^2) p^- = f_-,
\end{align}
and decompose the solution to $(\mcl_p - \gamma^2) p = f$ as 
\begin{align}
	p = p^- + p^c + \chi_+ p^+,
\end{align}
so that $p^c$ solves 
\begin{align}
	(\mcl_p - \gamma^2) p^c = \tilde{f}(\gamma) \label{e: p resolvent center equation}
\end{align}
with 
\begin{align}
	\tilde{f}(\gamma) = f_c + (\chi_+ - \chi_+^2) f - [\mcl_p^+, \chi_+] p^+ + (\mcl_p^+ - \mcl_p) (\chi_+ p^+) + (\mcl_p^- - \mcl_p) p^-,
\end{align}
where $[\mcl_p^+, \chi_+] = \mcl_p^+(\chi_+ \cdot) - \chi_+ \mcl_p^+ \cdot$ is the commutator. Since the coefficients of $\mcl_p$ attain their limits exponentially quickly as $x \to \pm \infty$ and the commutator $[\mcl_p^+, \chi_+]$ is compactly supported, $\tilde{f} (\gamma)$ is exponentially localized with rate uniform in $\gamma$ for $\gamma$ small. Utilizing also regularity of $p^+$ in $\gamma$, we see that $\tilde{f}(\gamma)$ has a well-defined limit at $\gamma = 0$, and we obtain the following detailed estimate, which is a special case of Lemma 3.7 of \cite{AveryScheelSelection}. 
\begin{lemma}\label{l: p resolvent tilde f localization}
	Let $\eta > 0$ be small and let $f \in L^1_{0,1} (\R)$. There exist positive constants $C$ and $\delta$ such that for $\gamma \in B(0, \delta)$ with $\Re \gamma \geq 0$, we have 
	\begin{align}
		\| \tilde{f}(\gamma) - \tilde{f}(0) \|_{L^1_{\mathrm{exp}, -\eta, \eta}} \leq C |\gamma| \| f \|_{L^1_{0,1}}. 
	\end{align}
\end{lemma}
With exponential localization of $\tilde{f}(\gamma)$, we solve \eqref{e: p resolvent center equation} via the far-field/core ansatz
\begin{align}
	p^c (x) = w (x) + b \chi_+ (x) e^{-\gamma x}. 
\end{align}
Inserting this ansatz into \eqref{e: p resolvent center equation} gives an equation
\begin{align}
	F(w, b; \gamma) = \tilde{f}(\gamma),
\end{align}
where 
\begin{align}
	F(w, b; \gamma) = \mcl_p w + b \mcl_p (\chi_+ e^{-\gamma \cdot}) - \gamma^2 (w + b \chi_+ e^{-\gamma \cdot}). 
\end{align}
Using Fredholm properties of $\mcl_p$ on exponentially weighted spaces, we obtain the following invertibility of $F$. See \cite{AveryScheel,AveryScheelSelection} for further details. We also carry out a similar argument in the following section in order to obtain estimates on $(\mcl_\psi-\gamma^2)^{-1}$. 
\begin{prop}\label{p: p resolvent invertibility}
	For $\eta > 0$ sufficiently small, there exists a $\delta> 0$ such that for each $\gamma \in B(0, \delta)$, the map 
	\begin{align}
		(w,b) \mapsto F(w, b; \gamma) : W^{2, 1}_{\mathrm{exp}, -\eta, \eta} (\R) \times \C \to L^1_{\mathrm{exp}, -\eta, \eta} (\R)
	\end{align}
	is invertible. We thereby denote the solution to $F( w, b; \gamma) = \tilde{f}$ by 
	\begin{align}
		(w,b) = (T(\gamma) \tilde{f}, B(\gamma) \tilde{f}),
	\end{align}
	with analytic maps 
	\begin{align}
		\gamma \mapsto (T(\gamma), B(\gamma)) : B(0,\delta) \to \mathcal{B}(L^1_{\mathrm{exp},-\eta,\eta}, W^{2,1}_{\mathrm{exp},-\eta, \eta}) \times \mathcal{B} (L^1_{\mathrm{exp},-\eta,\eta}, \C). 
	\end{align}
\end{prop}

We use this proposition to solve for $p^c$ and obtain the following estimates on the full resolvent. 

\begin{prop}\label{p: p resolvent L11 to L1 estimate}
	There exist positive constants $C$ and $\delta$ such that for any $f \in L^1_{0,1} (\R)$, we have 
	\begin{align}
		\| (\mcl_p-\gamma^2)^{-1} f \|_{W^{1,1}} \leq \frac{C}{|\gamma|} \| f \|_{L^1_{0,1}}
	\end{align}
	for all $\gamma \in B(0,\delta)$ with $\Re \gamma \geq \frac{1}{2} | \Im \gamma |$. 
\end{prop}
\begin{proof}
	The estimate on the derivative is strictly easier, so we focus only on estimating $(\mcl_p - \gamma^2)^{-1} f$ in $L^1$. The desired estimates for $\chi_+ p^+$ and $p^-$ follow from Lemma \ref{l: p resolvent right L11 to L1} and the fact that the spectrum of $\mcl_p^-$ does not contain the origin, so we only need to establish the estimate for $p^c$. By Proposition \ref{p: p resolvent invertibility}, we have 
	\begin{align*}
		p^c (\gamma) = T(\gamma) \tilde{f}(\gamma) + B(\gamma) \tilde{f}(\gamma) \chi_+ e^{-\gamma \cdot}. 
	\end{align*}
	Using the boundedness of $T(\gamma)$ together with the estimates in Lemma \ref{l: p resolvent tilde f localization}, we readily obtain 
	\begin{align*}
		\| T(\gamma) \tilde{f}(\gamma) \|_{L^1} \leq C \| T(\gamma)\|_{L^1_{\mathrm{exp}, -\eta, \eta} \to W^{2,1}_{\mathrm{exp}, -\eta, \eta}} \| \tilde{f} (\gamma) \|_{L^1_{\mathrm{exp},-\eta,\eta}} \leq C \| f \|_{L^1_{0,1}}. 
	\end{align*}
	For the other term, similarly using boundedness of $B(\gamma)$, we have 
	\begin{align*}
		\| B(\gamma) \tilde{f}(\gamma) \chi_+ e^{-\gamma \cdot} \|_{L^1} \leq C \| f \|_{L^1_{0,1}} \| \chi_+ e^{-\gamma \cdot}\|_{L^1}. 
	\end{align*}
	Since 
	\begin{align*}
		\| \chi_+ e^{-\gamma \cdot} \|_{L^1} \leq C \int_0^\infty e^{-c |\gamma| x} \, dx \leq \frac{C}{|\gamma|}
	\end{align*}
	for $\Re \gamma \geq \frac{1}{2} | \Im \gamma|$, we obtain 
	\begin{align*}
		\| p^c (\gamma) \|_{L^1} \leq \frac{C}{|\gamma|} \| f \|_{L^1_{0,1}},
	\end{align*}
	which completes the proof of the proposition. 
	
\end{proof}

\section{Resolvent estimates for $\mcl_\psi$}\label{s: psi resolvent estimates}
We use the same overall far-field/core approach to analyze the resolvent of $\mcl_\psi$. The main difference is that the spectrum of $\mcl_\psi^-$ is also marginally stable, in addition to that of $\mcl_\psi^+$, so that we have to take into account neutrally stable modes arising from behavior on the left as well as from the right. We first establish spectral stability, ruling out unstable eigenvalues and the possibility of an embedded ``eigenvalue'' (with a bounded eigenfunction) at the origin.  

\begin{lemma}\label{l: Lpsi spectral stability}
	The operator $\mcl_\psi : W^{2, \ell} (\R) \subseteq L^\ell (\R) \to L^\ell(\R)$ has no eigenvalues with $\Re \lambda \geq 0$ for any $1 \leq p \leq \infty$. Furthermore, there is no bounded solution to $\mcl_\psi \psi = 0$. 
\end{lemma}
\begin{proof}
	The linear equation $\mcl_\psi \psi = 0$ has a positive pointwise solution $\psi = \omega q_*$. A Sturm-Liouville argument therefore implies that there can be no eigenvalues with non-negative real part; see for instance \cite[proof of Theorem 5.5]{Sattinger}. Using standard theory of exponential dichotomies, one concludes that $\omega q_*$ is the only solution which is bounded for $x < 0$. Since $\omega q_*$ is unbounded on $x > 0$, we therefore see that there is no solution which is bounded on the whole real line, as desired. Alternatively, one can revert to $\phi$ coordinates and explicitly solve the resulting first-order ODE for $u = \phi_x$ to find the other linearly independent solution which is unbounded on $x < 0$.
\end{proof}

\subsection{Resolvent estimates for $\mcl_\psi^-$}
Here we study the resolvent for the limiting operator on the left in the $\psi$ linearization, $\mcl_\psi^- = \partial_{xx}+ 2 \partial_x$. This resolvent is given by 
\begin{align}
	(\mcl_\psi^- - \lambda)^{-1} f (x) = \int_\R G_\lambda^- (x-y) f(y) \, dy, 
\end{align}
where $G_\lambda^-$ is the resolvent kernel 
\begin{align}
	G_\lambda^- (x) = - \frac{1}{2 \sqrt{1+ \lambda}} \begin{cases}
		e^{\nu^-(\lambda) x}, & x \geq 0, \\
		e^{\nu^+(\lambda) x}, & x < 0,
	\end{cases} \label{e: G lambda minus}
\end{align}
with spatial eigenvalues $\nu^\pm(\lambda)$ given by 
\begin{align}
	\nu^\pm(\lambda) = - 1 \pm \sqrt{1 + \lambda}. 
\end{align}
Since the dispersion curve associated to $\mcl^-_\psi$ has no branch points, the resolvent kernel $G_\lambda^-$ is pointwise analytic in $\lambda$ in a neighborhood of the origin. Using the formula \eqref{e: G lambda minus} for the resolvent kernel, one readily obtains the following lemma. 
\begin{lemma}\label{l: psi left resolvent L1 to Linf boundedness}
	There exist positive constants $C$ and $\delta$ such that for any $f \in L^1 (\R)$, we have 
	\begin{align}
		\| (\mcl_\psi^- - \lambda)^{-1} f \|_{L^\infty} \leq C \| f \|_{L^1}
	\end{align}
	for all $\lambda \in B(0,\delta)$ to the right of the essential spectrum of $\mcl_\psi^-$. 
\end{lemma} 

We next analyze the regularity in $\lambda$ of the spatial derivative of $(\mcl_\psi^- - \lambda)^{-1} f$ in order to later characterize time decay of derivatives of $\psi$. 

\begin{lemma}\label{l: psi left resolvent derivative Linf estimate}
	There exist positive constants $C$ and $\delta$ and a bounded limiting operator $D^{-,0}_\psi : L^1 (\R) \to L^\infty(\R)$ such that for any $f \in L^1 (\R)$, we have 
	\begin{align}
		\| [ \partial_x (\mcl_\psi^- - \lambda)^{-1} - D^{-,0}_\psi ] f \|_{L^\infty} \leq C |\lambda| \| f \|_{L^1}, \label{e: psi left resolvent derivative l infty}
	\end{align}
	for all $\lambda \in B(0,\delta)$ to the right of the essential spectrum of $\mcl_{\psi}^-$. 
\end{lemma}
\begin{proof}
	For any $\lambda$ to the right of the essential spectrum of $\mcl_\psi^-$, we have 
	\begin{align}
		\partial_x (\mcl_\psi^- - \lambda)^{-1} f (x) &= \int_\R \partial_x G_\lambda^- (x-y) f(y) \, dy \\
		&= -\frac{1}{2\sqrt{1 + \lambda}} \left[ \int_{\{y \leq x\}} \nu^- (\lambda) e^{\nu^-(\lambda) (x-y)} f (y) \, dy + \int_{\{y \geq x\}} \nu^+(\lambda) e^{\nu^+(\lambda) (x-y)} f(y) \, dy \right]. \label{e: psi left resolvent derivative formula}
	\end{align}
	The coefficient $-\frac{1}{2 \sqrt{1+\lambda}}$ is analytic in $\lambda$, so we only need to estimate the integrals. For the first term, we have 
	\begin{align*}
		e^{\nu^-(\lambda) \xi} = e^{-2 \xi} + e^{-2\xi} (e^{(\nu^-(\lambda) + 2)\xi} - 1)
	\end{align*}
	If $|\nu^-(\lambda) + 2| |\xi| \leq 1$, then by Taylor expansion we have 
	\begin{align*}
		|e^{(\nu^-(\lambda) + 2)\xi} - 1| \leq C |\nu^-(\lambda) + 2| |\xi| \leq C |\lambda| |\xi| 
	\end{align*} 
	for $\lambda$ small. If on the other hand $|\nu^-(\lambda) + 2| |\xi| > 1$, then we simply use the additional factor of $e^{-2 \xi}$ to absorb the small exponential growth of $e^{(\nu^-(\lambda) + 2) \xi}$ possible for $\lambda$ small, so that 
	\begin{align*}
		|e^{-2 \xi} (e^{(\nu^-(\lambda)+2) \xi} - 1) | \leq C e^{-\xi} \leq  C | \nu^-(\lambda) + 2 | |\xi| |e^{-\xi} \leq C | \lambda| |\xi| e^{-\xi}. 
	\end{align*}
	Hence we obtain 
	\begin{align*}
		\sup_{x \in \R} \left| -\frac{1}{2 \sqrt{1+ \lambda}} \int_{\{y \leq x \}} \nu^-(\lambda) e^{\nu^-(\lambda) (x-y)} f(y) \, dy + \frac{1}{2} \int_{\{y \leq x \}} \nu^-(0) e^{\nu^-(0) (x-y)} f(y) \, dy \right| \leq C |\lambda| \| f \|_{L^1}. 
	\end{align*}
	For the other integral, we note that $\Re \nu^+ (\lambda) \geq 0$ for $\lambda$ to the right of the essential spectrum of $\mcl_\psi^-$, and $\nu^+(\lambda) = \mathrm{O}(\lambda)$ for $\lambda$ small, so 
	\begin{align*}
		\left| \int_{\{ y \geq x\}} \nu^+(\lambda) e^{\nu^+(\lambda) (x-y)} f(y) \, dy \right| \leq C |\nu^+(\lambda)| \int_{ \{ y \geq x \}} | f(y) | \, dy \leq C | \lambda| \| f \|_{L^1}
	\end{align*}
	for $\lambda$ small, completing the proof of \eqref{e: psi left resolvent derivative l infty}, with 
	\begin{align*}
		D^{-, 0}_\psi f (x) = \int_\R \partial_x G_0^-(x-y) f(y) \, dy. 
	\end{align*}
\end{proof}

\begin{lemma}\label{l: psi left resolvent derivative L1 estimate}
	There exist positive constants $C$, $c$, and $\delta$ and a contour $\Gamma$ given by 
	\begin{align}
		\Gamma = \{ ia - c a^2 : a \in [-\delta, \delta]\} \label{e: psi left resolvent derivative Gamma contour}
	\end{align}
	such that we have 
	\begin{align}
		\| \partial_x (\mcl_\psi^- - \lambda)^{-1} f \|_{L^1} \leq C \| f \|_{L^1}
	\end{align}
	for any $f \in L^1(\R)$ for $\lambda$ small and to the right of $\Gamma$. 
\end{lemma}
\begin{proof}
	We split the integral representation of the resolvent as in \eqref{e: psi left resolvent derivative formula}. Note that if we choose $c$ sufficiently small, $\Gamma$ is to the right of the essential spectrum of $\mcl_\psi^-$, touching it only at the origin. The estimate on the first integral in \eqref{e: psi left resolvent derivative formula} follows readily from the uniform exponential localization of $e^{\nu^-(\lambda) (x-y)}$ for $y \leq x$ and $\lambda$ to the right of the essential spectrum of $\mcl_\psi^-$, so we focus on the second integral. Using the change of variables $y = x-z$ there and then changing the order of integration, we have
	\begin{align*}
		\int_\R  \int_x^\infty \nu^+(\lambda) e^{\nu^+(\lambda)(x-y)} f(y) \, dy \, dx &= \nu^+(\lambda) \int_\R \int_{-\infty}^0 e^{\nu^+(\lambda) z} f(x-z) \, dz \, dx \\ 
		&= \nu_+(\lambda) \left( \int_\R f(\xi) \, d\xi \right) \left( \int_{-\infty}^0 e^{\nu^+(\lambda) z} \, dz \right) \\
		&= \frac{\nu^+(\lambda)}{\nu^+(\lambda)} \int_\R f(\xi) \, d \xi
	\end{align*}
	provided $\lambda$ is to the right of the essential spectrum of $\mcl_\psi^-$, from which the lemma follows. 
\end{proof}

\begin{lemma}\label{l: psi left resolvent L1 to Linf-1 estimate}
	There exist positive constants $C$ and $\delta$ and a bounded limiting operator $R_\psi^{-,0} : L^1 (\R) \to L^\infty(\R)$ such that for any $f \in L^1 (\R)$ with $\mathrm{supp} (f) \subseteq (-\infty, 0]$, we have 
	\begin{align}
		\left\|\chi_- \left[ (\mcl_\psi^- - \lambda)^{-1} - R_\psi^{-,0} \right] f \right\|_{L^\infty_{-1,0}} \leq C |\lambda| \| f \|_{L^1}
	\end{align}
	for any $\lambda \in B(0,\delta)$ to the right of the essential spectrum of $\mcl_\psi^-$. 
\end{lemma}
\begin{proof}
Using the formula \eqref{e: G lambda minus} for the resolvent kernel, we have 
\begin{align*}
	(\mcl_\psi^- - \lambda)^{-1} f (x) = - \frac{1}{2 \sqrt{1 + \lambda}} \left( \int_{\{y \leq x\}} e^{\nu^-(\lambda) (x-y)} f(y) \, dy + \int_{\{y \geq x\}} e^{\nu^+(\lambda) (x-y)} f(y) \, dy \right).
\end{align*}
Regularity in $\lambda$ of the first integral may be established as in the proof of Lemma \ref{l: psi left resolvent derivative Linf estimate}, so we focus on the second term. Exploiting the fact that $\mathrm{supp} (f) \subseteq (-\infty, 0]$, we have 
\begin{align*}
	\chi_- (x) \int_{\{ y \geq x \}} e^{\nu^+(\lambda) (x-y)} f(y) \, dy = \chi_- (x) \int_x^0 e^{\nu^+(\lambda) (x-y)} f(y) \, dy.
\end{align*}
For $x \leq y \leq 0$ and for $\lambda$ to the right of the essential spectrum of $\mcl_\psi^-$, we have $\Re [\nu^+(\lambda) (x-y)] \leq 0$, and so 
\begin{align*}
	| e^{\nu^+(\lambda)(x-y)} - 1 | \leq C | \nu^+(\lambda)| | x-y| \leq C |\lambda| (|x| +|y|) \leq C |\lambda| |x|. 
\end{align*}
Hence we have 
\begin{align*}
	\left|\chi_- (x) \int_x^0 \left( e^{\nu^+(\lambda) (x-y)} - 1\right) f(y) \, dy\right| \leq C |\lambda| \chi_- (x)\langle x \rangle \int_{x}^0 |f(y)| \,dy \leq C |\lambda| \chi_- (x) \langle x \rangle \| f \|_{L^1},
\end{align*}
from which the lemma readily follows. 
\end{proof}

Next we characterize regularity in a full neighborhood of $\lambda = 0$ given an extra exponentially localized coefficient, which is needed to reconcile with regularity of the right resolvent in our far-field/core argument in the following section.

\begin{lemma}\label{l: psi left resolvent exponential localization}
	Fix $\eta > 0$ and $1 \leq \ell \leq \infty$. There exist positive constants $C$ and $\delta$ such that for any $f \in L^1 (\R)$ with $\mathrm{supp} f \subseteq (-\infty, 0]$, we have for $\lambda \in B(0,\delta)$
	\begin{align}
		\left\| e^{\eta \cdot} \chi_- \left[ (\mcl_{\psi}^- - \lambda)^{-1} - R_{\psi}^{-, 0} \right] f \right\|_{L^\ell} \leq C | \lambda| \| f \|_{L^1}.
	\end{align}
\end{lemma}
\begin{proof}
	We focus first on the case $\ell = \infty$. As in the proof of Lemma \ref{l: psi left resolvent L1 to Linf-1 estimate}, we split the integral representation of the resolvent into two terms and focus on the integral from $x$ to $0$. We claim that for $x \leq y \leq 0$, we have 
	\begin{align}
		| e^{\eta x} (e^{\nu^+(\lambda) (x-y)} - 1 ) | \leq C |\lambda| \label{e: psi left resolvent kernel exponential localized estimate}
	\end{align}
	for $\lambda$ sufficiently small. To prove this, we first consider the case where $| \nu^+(\lambda) (x-y)| \leq 1$. In this case, we have by Taylor's theorem 
	\begin{align*}
		| e^{\eta x} (e^{\nu^+(\lambda) (x-y)} - 1 ) | \leq C e^{\eta x} | x-y| | \nu^+(\lambda) | \leq C | x |e^{\eta x} |\lambda| \leq C |\lambda|,
	\end{align*}
	since $|y| \leq |x|$ for $x \leq y \leq 0$. On the other hand, if $|\nu^+(\lambda) (x-y)| > 1$ and $x \leq y \leq 0$, we have 
	\begin{align*}
		| e^{\eta x} (e^{\nu^+(\lambda) (x-y)} - 1 ) | \leq e^{\frac{\eta}{2} x} e^{\frac{\eta}{2} x} \left( | e^{\nu^+(\lambda)(x-y)} | + 1 \right) \leq e^{\frac{\eta}{2} x} e^{\frac{\eta}{2} y} \left( e^{\Re \nu^+(\lambda) x} e^{-\Re \nu^+(\lambda) y} + 1 \right). 
	\end{align*}
	If $\lambda$ is sufficiently small relative to $\eta$, then the small exponential growth of $e^{\Re \nu^+(\lambda) x}$ and $e^{-\Re \nu^+(\lambda) y}$ can be absorbed into the factors of $e^{\frac{\eta}{4} x }$ and $e^{\frac{\eta}{2} y}$. Hence we obtain 
	\begin{align*}
			| e^{\eta x} (e^{\nu^+(\lambda) (x-y)} - 1 ) | \leq C e^{\frac{\eta}{4} x} \leq C | \nu^+(\lambda) (x-y)| e^{\frac{\eta}{4} x} \leq C |\nu^+(\lambda)| | x | e^{\frac{\eta}{4} x} \leq C |\lambda|,
	\end{align*}
	which completes the proof of \eqref{e: psi left resolvent kernel exponential localized estimate}. By this estimate, we have 
	\begin{align*}
		e^{\eta x} \chi_- (x) \left| \int_x^0 (e^{\nu^+(\lambda) (x-y)}-1) f(y) \, dy \right|\leq C | \lambda| \int_x^0 | f(y) | \, dy \leq C |\lambda| \| f \|_{L^1},
	\end{align*}
	which completes the proof of the lemma in the case $\ell = \infty$. To handle $1 \leq \ell < \infty$, we simply write 
	\begin{align*}
		\left\| e^{\eta \cdot} \chi_- \left[ (\mcl_{\psi}^- - \lambda)^{-1} - R_{\psi}^{-, 0} \right] f \right\|_{L^\ell} &\leq C \| \chi_- e^{\frac{\eta}{2} x} \|_{L^\ell} \| \left\| e^{\frac{\eta}{2} \cdot} \chi_- \left[ (\mcl_{\psi}^- - \lambda)^{-1} - R_{\psi}^{-, 0} \right] f \right\|_{L^\infty} \\
		&\leq C \left\| e^{\frac{\eta}{2} \cdot} \chi_- \left[ (\mcl_{\psi}^- - \lambda)^{-1} - R_{\psi}^{-, 0} \right] f \right\|_{L^\infty},
	\end{align*}
	and repeat the above argument with $\tilde{\eta} = \eta/2$. 
\end{proof}

\subsection{Resolvent estimates for $\mcl_\psi$}
As in Section \ref{s: p full resolvent estimates}, we solve the equation $(\mcl_\psi - \gamma^2) \psi = f$ via a far-field/core decomposition. We again decompose $f$ using the partition of unity introduced in Section \ref{s: p full resolvent estimates} as 
\begin{align*}
	f = \chi_- f + \chi_c f + \chi_+ f =: f_- + f_c + f_+. 
\end{align*}
We let $f_+^\mathrm{odd} (x) = f_+ (x) - f_+ (-x)$ be the odd extension of $f_+$, and let $\psi^+$ solve 
\begin{align}
	(\mcl_\psi^+ - \gamma^2) \psi^+ = f_+^\mathrm{odd}
\end{align}
in order to take advantage of the improved properties of $(\mcl_\psi^+ - \gamma^2)^{-1} = (\partial_{xx} - \gamma^2)^{-1}$ on odd functions. We let $\psi^-$ solve 
\begin{align*}
	(\mcl_\psi^- - \gamma^2) \psi^- = f_-,
\end{align*}
and then decompose the solution $\psi$ to $(\mcl - \gamma^2) \psi = f$ as 
\begin{align*}
	\psi = \chi_- \psi^- + \psi^c + \chi_+ \psi^+. 
\end{align*}
As a result, $\psi^c$ solves 
\begin{align}
	(\mcl_\psi - \gamma^2) \psi^c = \tilde{f} (\gamma), \label{e: psi c eqn}
\end{align}
where 
\begin{multline}
	\tilde{f} (\gamma) = f_c + (\chi_- - \chi_-^2) f + (\chi_+ - \chi_+^2) f - [\mcl_\psi^-, \chi_-] \psi^- + (\mcl_\psi^- - \mcl_\psi) (\chi_- \psi^-) - [\mcl_\psi^+, \chi_+] \psi^+ \\ + (\mcl_\psi^+ - \mcl_\psi) (\chi_+ \psi^+). \label{e: psi resolvent tilde f def}
\end{multline}

As in the previous section, we start by establishing regularity of $\tilde{f}$ in $\gamma$ in appropriate exponentially weighted spaces. 

\begin{lemma}\label{l: psi resolvent tilde f estimates}
	There exist positive constants $C$ and $\delta$ such that for any $\gamma \in B(0,\delta)$ with $\Re \gamma \geq 0$, we have 
	\begin{align}
		\| \tilde{f} (\gamma) - \tilde{f}(0) \|_{L^\ell_{\mathrm{exp}, -\eta, \eta}} \leq C |\gamma| \| f \|_{L^1_{0,1}} \label{e: psi resolvent tilde f lipschitz estimate}
	\end{align}
	for any $f \in L^1_{0,1} (\R)$, for $\ell = 1$ or $\infty$ as well as 
	\begin{align}
		\| \tilde{f}(\gamma) \|_{L^\infty_{\mathrm{exp}, -\eta, \eta}} \leq C \left( \| f \|_{L^1} + \| f \|_{L^\infty} \right) \label{e: psi resolvent tilde f boundedness}
	\end{align}
	for any $f \in L^1 \cap L^\infty (\R)$, and 
	\begin{align}
		\| \tilde{f} (\gamma) \|_{L^1_{\mathrm{exp},-\eta,\eta}} \leq C \| f \|_{L^1} \label{e: psi resolvent tilde f L1 estimate}
	\end{align}
	for any $f \in L^1 (\R)$. 
\end{lemma}
\begin{proof}
	First we prove \eqref{e: psi resolvent tilde f lipschitz estimate}. The dependence on $\gamma$ enters through $\psi^+$ and $\psi^-$. The coefficients of $\psi^+$ and its derivatives in \eqref{e: psi resolvent tilde f def} are all supported on $x > 0$ and exponentially localized with a rate uniform in $\gamma$ for $\gamma$ small. It follows from \cite[Proposition 3.1]{AveryScheelSelection} that these terms are Lipschitz in $\gamma$ in the desired region. 
	
	For the terms involving $\psi^-$, expansions to order $\lambda = \gamma^2$ are guaranteed by Lemma \ref{l: psi left resolvent exponential localization} in a full neighborhood of $\lambda = 0$ thanks to the strongly localized coefficients, so in particular for $\Re \gamma \geq 0$, as desired. 
	
	The estimate \eqref{e: psi resolvent tilde f boundedness} follows similarly, but with localization of terms involving $\psi^+$ obtained from Lemma \ref{l: right resolvent exponentially localized boundedness}. The term $\| f \|_{L^\infty}$ on the right hand side is needed to control $\| f_c + (\chi_- - \chi_-^2) f + (\chi_+ - \chi_+^2) f \|_{L^\infty_{\exp, -\eta, \eta}}$ in terms of spatial regularity. The estimate \eqref{e: psi resolvent tilde f L1 estimate} is similar. 
\end{proof}

We solve \eqref{e: psi c eqn} by taking advantage of this exponential localization of $\tilde{f}$ and the Fredholm properties of $\mcl_\psi$ on exponentially weighted spaces. To carry this out, we make the ansatz 
\begin{align}
	\psi^c(x) = \beta_- \chi_- (x) e^{\nu^+(\gamma^2) x} + v(x) + \beta_+ \chi_+ (x) e^{-\gamma x},
\end{align}
requiring $v$ to be exponentially localized. Inserting this ansatz into \eqref{e: psi c eqn} leads to an equation 
\begin{align}
	F(v, \beta_-, \beta_+; \gamma) = \tilde{f}(\gamma),
\end{align}
where 
\begin{align}
	F(v, \beta_-, \beta_+; \gamma) = (\mcl_\psi - \gamma^2) v + \beta_- (\mcl_\psi - \gamma^2) (\chi_- e^{\nu^+(\gamma^2) \cdot}) + \beta_+ (\mcl_\psi - \gamma^2) (\chi_+ e^{-\gamma \cdot}). 
\end{align}

For $\eta > 0$, we let $(X_\eta, Y_\eta)$ denote either pair of spaces 
\begin{align}
	X_\eta = L^1_{\mathrm{exp}, -\eta, \eta} (\R), \, Y_\eta = W^{2,1}_{\mathrm{exp}, -\eta, \eta} (\R) \quad \text{or} \quad X_{\eta} = L^\infty_{\mathrm{exp}, -\eta, \eta} (\R), \, Y_\eta = W^{2, \infty}_{\mathrm{exp}, -\eta, \eta} (\R). \label{e: X eta Y eta def}
\end{align}

\begin{lemma}\label{l: psi resolvent fredholm properties}
	For $\eta > 0$ sufficiently small, the operator $\mcl_\psi: Y_\eta \to X_\eta$ is Fredholm with index $-2$ for either pair of spaces $(X_\eta, Y_\eta)$ defined in \eqref{e: X eta Y eta def}. 
\end{lemma}
\begin{proof}
	The Fredholm index may be computed from the asymptotic dispersion relations,
	\begin{align}
		d^+_\psi (\lambda, \nu) &= \nu^2 - \lambda, \\
		d^-_\psi (\lambda, \nu) &= \nu^2 + 2 \nu - \lambda,
	\end{align}
	obtained from substituting $\psi = e^{\lambda t + \nu x}$ into $\psi_t = \mcl_\psi^\pm \psi$, as follows; see for instance \cite{KapitulaPromislow, FiedlerScheel} for background. To take into account the effect of the exponential weight, we compute the roots to $\nu \mapsto d^+_\psi (\lambda, \nu-\eta)$ for $\eta > 0$ small, and find a double root at $\nu = \eta > 0$. The roots of $d^-_\psi(0, \nu + \eta)$ are $\nu = - \eta$ and $\nu = - \eta - 2$, each negative for $\eta > 0$. The fact that none of these roots are zero implies that $\mcl_\psi$ is Fredholm on $X_\eta$, and its Fredholm index is given by the difference of the Morse indices,
	\begin{align*}
		\mathrm{ind} (\mcl_\psi) = \left(\# \text{ of positive roots of } d^-_\psi \right) - \left(\# \text{ of positive roots of } d^+_\psi \right) = 0 - 2 = -2,
	\end{align*}
	as desired. 
\end{proof}

\begin{lemma}
	For $\eta > 0$ sufficiently small and for either pair of spaces $(X_\eta, Y_\eta)$ defined in \eqref{e: X eta Y eta def} there exists a $\delta > 0$ such that $F : Y_\eta \times \C^2 \times B(0,\delta) \to X_\eta$ is well defined and the mapping 
	\begin{align}
		\gamma \mapsto F(\cdot, \cdot, \cdot; \gamma) : B(0, \delta) \to \mathcal{B}(Y_\eta \times \C^2, X_\eta) 
	\end{align}
	is analytic in $\gamma$. 
\end{lemma}
\begin{proof}
	We rewrite $F$ as 
	\begin{multline*}
		F(v, \beta_-, \beta_+; \gamma) = (\mcl_\psi - \gamma^2) v + \beta_- \left( (\mcl_\psi - \mcl_\psi^-) (\chi_- e^{\nu^+(\gamma^2) \cdot} ) + [\mcl_\psi^-, \chi_-] e^{\nu^+(\gamma^2)\cdot}\right) \\ + \beta_+ \left( (\mcl_\psi - \mcl_\psi^+) (\chi_+ e^{-\gamma \cdot}) + [\mcl_\psi^+, \chi_+] e^{-\gamma \cdot} \right),
	\end{multline*}
	taking advantage of the fact that $(\mcl_\psi^- - \gamma^2) (e^{\nu^+(\gamma^2) \cdot}) = (\mcl_\psi^+ - \gamma^2) e^{-\gamma \cdot} = 0$. The terms involving $(\mcl_\psi - \mcl_\psi^\pm)$ or commutators with the cutoff functions are all exponentially localized with uniform rate, so $F$ preserves exponential localization and hence is well defined on the above spaces, as desired. Analyticity follows from the fact that $\nu^+(\gamma^2)$ is analytic; for more details, see for instance \cite[proof of Lemma 3.9]{AveryScheelSelection}.
\end{proof}

\begin{corollary}\label{c: psi resolvent farfield core invertibility}
	For $\eta > 0$ sufficiently small and for either pair of spaces $(X_\eta, Y_\eta)$ in \eqref{e: X eta Y eta def}, there exists a $\delta > 0$ such that for $\gamma \in B(0,\delta)$, the map 
	\begin{align}
		(v, \beta_-, \beta_+) \mapsto F(v, \beta_-, \beta_+; \gamma) : Y_\eta \times \C^2 \to X_\eta 
	\end{align}
	is invertible. We denote the solution to $F(v, \beta_-, \beta_+; \gamma) = \tilde{f}$ by 
	\begin{align}
		(v, \beta_-, \beta_+) = (T(\gamma) \tilde{f}, B_- (\gamma) \tilde{f}, B_+(\gamma) \tilde{f})
	\end{align}
	with analytic maps 
	\begin{align}
		\gamma \mapsto (T(\gamma), B_- (\gamma), B_+(\gamma)) : B(0,\delta) \to \mathcal{B}(X_\eta, Y_\eta) \times \mathcal{B} (X_\eta, \C) \times \mathcal{B}(X_\eta, \C). 
	\end{align}
\end{corollary}

\begin{proof}
	Note that $F$ is linear in $(v, \beta_-, \beta_+)$. Since $D_v F(v, \beta_-, \beta_+; 0)$ = $\mcl_\psi$ is Fredholm index -2 on $Y_\eta$ by Lemma \ref{l: psi resolvent fredholm properties}, it follows from the Fredholm bordering lemma that $(v, \beta_-, \beta_+) \mapsto F(v, \beta_-, \beta_+; \gamma)$ is Fredholm index 0 for $\gamma$ sufficiently small. Lemma \ref{l: Lpsi spectral stability} implies that this map has trivial kernel at $\gamma = 0$, since a nontrivial kernel would give rise to a bounded solution to $\mcl_\psi \psi = 0$. The result then follows from the implicit function theorem. 
\end{proof}

Having solved $(\mcl_\psi - \gamma^2) \psi = f$ in a neighborhood of $\gamma = 0$ in the preceding corollary, we now analyze the regularity of the solution in $\gamma$ near the essential spectrum. The following proposition will be used in Section \ref{s: linear estimates} to establish diffusive decay of $\| e^{\mcl_\psi t }\|_{L^1 \to L^\infty}$. 

\begin{prop}\label{p: psi resolvent L1 to L inf estimate}
	There exist positive constants $C$ and $\delta$ such that for any $f \in L^1(\R) \cap L^\infty(\R)$, the solution $\psi$ to $(\mcl_\psi - \gamma^2)\psi = f$ admits a decomposition $\psi = \psi^\mathrm{left} + \psi^{\mathrm{right}}$, with
	\begin{align}
		\| \psi^\mathrm{left}(\gamma) \|_{L^\infty} \leq C \left( \| f \|_{L^1} + \| f \|_{L^\infty} \right)
	\end{align}
	for all $\gamma \in B(0, \delta)$ such that $\gamma^2$ is to the right of the spectrum of $\mcl_\psi^-$, and 
	\begin{align}
		\| \psi^\mathrm{right} (\gamma) \|_{L^\infty} \leq \frac{C}{|\gamma|} \left( \| f \|_{L^1} + \| f \|_{L^\infty} \right)
	\end{align}
	for all $\gamma \in B(0, \delta)$ with $\Re \gamma \geq \frac{1}{2} | \Im \gamma|$. Furthermore, for each fixed $x \in \R$, the maps $\gamma \mapsto (\psi^\mathrm{left}(x; \gamma), \psi^\mathrm{right}(x; \gamma))$ are analytic for $\Re \gamma > 0$. 
\end{prop}
\begin{proof}
	With notation as in the far-field/core decomposition carried out above, define 
	\begin{align}
		\psi^\mathrm{left} (\gamma) = \chi_- \psi^- (\gamma) + B_- (\gamma) \tilde{f}(\gamma) \chi_- e^{\nu^+(\gamma^2) \cdot}. \label{e: psi resolvent psi left def}
	\end{align}
	The term $\chi_- \psi^- (\gamma)$ satisfies the desired estimate by Lemma \ref{l: psi left resolvent L1 to Linf boundedness}. For the second term, we have by Corollary \ref{c: psi resolvent farfield core invertibility}
	\begin{align*}
		\| B_- (\gamma) \tilde{f}(\gamma) \chi_- e^{\nu^+(\gamma^2)\cdot} \|_{L^\infty} &\leq C \| B_- (\gamma) \|_{L^\infty_{\mathrm{exp},-\eta,\eta} \to \C} \| \tilde{f} (\gamma) \|_{L^\infty_{\mathrm{exp}, -\eta, \eta} } \| \chi_- e^{\nu^+(\gamma^2) \cdot} \|_{L^\infty} \\
		 &\leq C \| f \|_{L^1 \cap L^\infty} \| \chi_- e^{\nu^+(\gamma^2) \cdot} \|_{L^\infty},
	\end{align*}
	using Corollary \ref{c: psi resolvent farfield core invertibility} to estimate the operator norm of $B_- (\gamma)$ and Lemma \ref{l: psi resolvent tilde f estimates} to control the norm of $\tilde{f}(\gamma)$. For $\gamma^2$ to the right of the essential spectrum of $\mcl_\psi^-$, $\chi_- e^{\nu^+(\gamma^2) \cdot}$ is bounded, so $\psi^\mathrm{left}$ satisfies the desired estimate. 
	
	We collect the remaining terms in the far-field/core decomposition in $\psi^\mathrm{right}$, including the exponentially localized ``center'' term, so 
	\begin{align}
		\psi^\mathrm{right} (\gamma) = \chi_+ \psi^+ (\gamma) + T(\gamma) \tilde{f} (\gamma) + B_+(\gamma) \tilde{f}(\gamma) \chi_+ e^{-\gamma \cdot}. \label{e: psi resolvent psi right def}
	\end{align}
	The term $\chi_+ \psi^+(\gamma)$ satisfies the desired estimate by Lemma \ref{l: right resolvent L1 to Linf estimate}, and the center term $T(\gamma)\tilde{f}(\gamma)$ is controlled in $L^\infty$ by $\| f \|_{L^1 \cap L^\infty}$ thanks to Corollary \ref{c: psi resolvent farfield core invertibility} and Lemma \ref{l: psi resolvent tilde f estimates}. Also using Corollary \ref{c: psi resolvent farfield core invertibility} and Lemma \ref{l: psi resolvent tilde f estimates} to control $B_+(\gamma) \tilde{f}(\gamma)$, we obtain 
	\begin{align*}
		\| B_+(\gamma) \tilde{f}(\gamma) \chi_+ e^{-\gamma \cdot}\|_{L^\infty} \leq C \| f \|_{L^1 \cap L^\infty}
	\end{align*}
	for $\Re \gamma \geq 0$, so in particular $\psi^\mathrm{right}$ satisfies the desired estimate. 
	
	The pointwise analyticity follows from the analyticity in $\gamma$ from Corollary \ref{c: psi resolvent farfield core invertibility} as well as the representation formulas for $\psi^+$ and $\psi^-$ via integration against the resolvent kernels. 
\end{proof}

The next proposition takes advantage of the outward transport on the left to extract faster temporal decay, characterized here by higher regularity in $\gamma$, by measuring the solution in a norm which allows algebraic growth on the left. 

\begin{prop}\label{p: psi resolvent L1 to Linf-10}
	There exist positive constants $C$ and $\delta$ such that for any $f \in L^1_{0,1} (\R) \cap L^\infty(\R)$, the solution to $(\mcl_\psi - \gamma^2) \psi = f$ admits a decomposition $\psi = \tilde{\psi}^\mathrm{left}+ \tilde{\psi}^\mathrm{right}$, with 
	\begin{align}
		\| \tilde{\psi}^\mathrm{left} (\gamma) - \tilde{\psi}^\mathrm{left}(0) \|_{L^\infty_{-1, 0}} \leq C | \gamma|^2 \left( \| f \|_{L^1_{0,1}} + \| f \|_{L^\infty} \right) 
	\end{align}
	for all $\gamma \in B(0,\delta)$ with $\gamma^2$ to the right of the essential spectrum of $\mcl_\psi^-$, and 
	\begin{align}
		\| \tilde{\psi}^\mathrm{right} (\gamma) \|_{L^\infty_{-1,0}} \leq C \left( \| f \|_{L^1_{0,1}} + \| f \|_{L^\infty} \right)
	\end{align}
	for all $\gamma \in B(0,\delta)$ with $\Re \gamma \geq \frac{1}{2} | \Im \gamma|$. 
\end{prop}
\begin{proof}
	We modify the decomposition from Proposition \ref{p: psi resolvent L1 to L inf estimate} slightly, this time writing
	\begin{align}
		\tilde{\psi}^\mathrm{left}(\gamma) = \chi_- \psi^- (\gamma) + B_- (\gamma) \tilde{f}(\gamma) \chi_- (e^{\nu^+(\gamma^2) \cdot} - 1),
	\end{align}
	and 
	\begin{align}
		\tilde{\psi}^\mathrm{right}(\gamma) = B_- (\gamma) \tilde{f}(\gamma) \chi_- + \chi_+ \psi^+(\gamma) + T(\gamma) \tilde{f}(\gamma) + B_+(\gamma) \tilde{f}(\gamma) \chi_+ e^{-\gamma \cdot}. 
	\end{align}
	By Lemma \ref{l: psi left resolvent L1 to Linf-1 estimate}, we have 
	\begin{align*}
		\| \chi_- \psi^- (\gamma) - \chi_- \psi^- (0) \|_{L^\infty_{-1, 0}} \leq C | \gamma|^2 \| f \|_{L^1_{0,1}},
	\end{align*}
	since $f_-$ is supported on $(-\infty,0]$ by construction. Since $\Re \nu^+(\gamma^2) x \leq 0$ for $\gamma^2$ to the right of the essential spectrum of $\mcl_\psi^-$, we have $| e^{\nu^+(\gamma^2) x} - 1 | \leq C | \nu^+(\gamma^2) | | x | \leq C |\gamma|^2 |x|$ for $\gamma$ small in this region and $x \leq 0$. By Corollary \ref{c: psi resolvent farfield core invertibility} and Lemma \ref{l: psi resolvent tilde f estimates}, $B_- (\gamma) \tilde{f}(\gamma)$ is controlled by $\| f \|_{L^\infty}$, uniformly in $\gamma$ in this region, and hence we conclude 
	\begin{align*}
		\| \tilde{\psi}^\mathrm{left}(\gamma) - \tilde{\psi}^\mathrm{left}(0) \|_{L^\infty_{-1, 0}} \leq C | \gamma|^2 \left( \| f \|_{L^1_{0,1}} + \| f \|_{L^\infty} \right), 
	\end{align*}
	as desired. 
	
	We now estimate $\tilde{\psi}^\mathrm{right}$. By Corollary \ref{c: psi resolvent farfield core invertibility} and Lemma \ref{l: psi resolvent tilde f estimates}, we have 
	\begin{align*}
		\| B_-(\gamma) \tilde{f} (\gamma) \chi_- \|_{L^\infty_{-1, 0}} \leq C \| B_-(\gamma)\|_{L^\infty_{\mathrm{exp},-\eta, \eta} \to \C} \| \tilde{f}(\gamma) \|_{L^\infty_{\mathrm{exp},-\eta,\eta}} \leq C \| f \|_{L^\infty} 
	\end{align*}
	for $\gamma$ small with $\Re \gamma \geq 0$. For the term $\chi_+ \psi^+(\gamma)$, we have 
	\begin{align*}
		\chi_+ (x) \psi^+(x; \gamma) = \chi_+ (x) \int_0^\infty G_\gamma^\mathrm{odd} (x,y) f_+^\mathrm{odd} (y) \, dy. 
	\end{align*}
	By the pointwise estimate \eqref{e: right resolvent G odd pointwise estimate} on $G_\mathrm{\gamma}^\mathrm{odd}$, we have 
	\begin{align*}
		\sup_{x \in \R} | \chi_+ (x) \psi^+(x; \gamma) | \leq \chi_+ (x) \int_0^\infty e^{-c |\gamma| |x-y|} \langle y \rangle |f_+^\mathrm{odd}(y)| \, dy \leq C \| f \|_{L^1_{0,1}},
	\end{align*}
	as desired. For the remaining terms in $\tilde{\psi}^\mathrm{right}$, we have
	\begin{align*}
		\| T(\gamma) \tilde{f}(\gamma) + B_+ (\gamma) \tilde{f}(\gamma) \chi_+ e^{-\gamma \cdot} \|_{L^\infty_{-1, 0}} &\leq C \left( \| f \|_{L^1_{0,1}} + \| f \|_{L^\infty} \right) ( 1 + \| \chi_+ e^{-\gamma \cdot} \|_{L^\infty_{-1,0}} ) \\
		&\leq C  \left( \| f \|_{L^1_{0,1}} + \| f \|_{L^\infty} \right)
	\end{align*}
	for $\gamma$ small with $\Re \gamma \geq \frac{1}{2} | \Im \gamma|$, using Corollary \ref{c: psi resolvent farfield core invertibility} and Lemma \ref{l: psi resolvent tilde f estimates} to control $T(\gamma) \tilde{f}(\gamma)$ and $B_+(\gamma) \tilde{f}(\gamma)$. This completes the proof of the proposition. 
\end{proof}

We now establish improved regularity in $\gamma$ of spatial derivatives of $\psi$, which will eventually translate into improved temporal decay. First, we characterize this regularity when measuring the derivative in $L^1$.

\begin{prop}\label{p: psi resolvent derivative L1 to L1 estimate}
	There exist constants $C$ and $\delta > 0$ such that for all $f \in L^1(\R)$, decomposing the solution to $(\mcl_\psi - \gamma^2) \psi = f$ as in Proposition \ref{p: psi resolvent L1 to L inf estimate}, we have 
	\begin{align}
		\| \partial_x \psi^\mathrm{left} (\gamma) \|_{L^1} \leq C \| f \|_{L^1}
	\end{align}
	for all $\gamma \in B(0, \delta)$ with $\gamma^2$ to the right of the contour $\Gamma$ defined in \eqref{e: psi left resolvent derivative Gamma contour}, and
	\begin{align}
		\| \partial_x \psi^\mathrm{right} (\gamma) \|_{L^1} \leq \frac{C}{|\gamma|} \| f \|_{L^1}
	\end{align}
	for all $\gamma \in B(0, \delta)$ with $\Re \gamma \geq \frac{1}{2} | \Im \gamma|$. 
\end{prop}
\begin{proof}
	First we estimate $\partial_x \psi^\mathrm{left}$. We have 
	\begin{align*}
		\partial_x (\chi_- \psi^-(\gamma)) = \chi_-' \psi^- (\gamma) + \chi_- \partial_x \psi^- (\gamma).
	\end{align*}
	The second term satisfies the desired estimate by Lemma \ref{l: psi left resolvent derivative L1 estimate}. For the first term, we have 
	\begin{align*}
		\| \chi_-' \psi^-(\gamma) \|_{L^1} \leq \| \chi_-' \|_{L^1} \| \psi^-(\gamma) \|_{L^\infty} \leq C \| f \|_{L^1}
	\end{align*}
	by Lemma \ref{l: psi left resolvent L1 to Linf boundedness}. For the other part of $\partial_x \psi^\mathrm{left}$, we have by Corollary \ref{c: psi resolvent farfield core invertibility} and Lemma \ref{l: psi resolvent tilde f estimates}
	\begin{align*}
		\| \partial_x (B_-(\gamma) \tilde{f}(\gamma) \chi_- e^{\nu^+(\gamma^2) \cdot}) \|_{L^1} \leq C \| f \|_{L^1} \left( \| \chi_-' e^{\nu^+(\gamma^2) \cdot} \|_{L^1} + |\nu^+(\gamma^2)| \| \chi_- e^{\nu^+(\gamma^2) \cdot} \|_{L^1} \right).
	\end{align*}
	For the first term in parenthesis, we have 
	\begin{align*}
		\| \chi_-' e^{\nu^+(\gamma^2) \cdot} \|_{L^1} \leq C \| \chi_-' \|_{L^1} \| \chi_- e^{\nu^+(\gamma^2) \cdot} \|_{L^\infty} \leq C 
	\end{align*}
	for $\gamma^2$ to the right of the essential spectrum of $\mcl_\psi^-$. The second term in the parenthesis is bounded by the argument in the proof of Lemma \ref{l: psi left resolvent derivative L1 estimate}. Hence $\partial_x \psi^\mathrm{left}$ satisfies the desired estimate. 
	
	Now we estimate $\partial_x \psi^\mathrm{right}$. First, we write $\psi^+$ as 
	\begin{align}
		\psi^+(x;\gamma) = \frac{1}{2 \gamma} \int_\R e^{-\gamma |x-y|} f_+^\mathrm{odd}(y) \, dy,
	\end{align}
	using the integral kernel for $(\partial_{xx}-\gamma^2)^{-1}$ on the real line without specifically taking advantage of odd data. We then have by Young's inequality
	\begin{align*}
		\| \partial_x \psi^+(\gamma) \|_{L^1} \leq C \| e^{-\gamma |\cdot|} \|_{L^1} \| f_+^\mathrm{odd} \|_{L^1} \leq \frac{C}{|\gamma|} \| f \|_{L^1}
	\end{align*}
	provided $\Re \gamma \geq \frac{1}{2} | \Im \gamma|$. We write 
	\begin{align*}
		\partial_x (\chi_+ \psi^+) = \chi_+' \psi^+ + \chi_+ \partial_x \psi^+.
	\end{align*}
	The second term satisfies the desired estimate by the above argument. For the first term, we have 
	\begin{align*}
		\| \chi_+' \psi^+ \|_{L^1} \leq C \| e^{\eta \cdot} \chi_+' \|_{L^1} \| e^{-\eta \cdot} \chi_+ \psi^+ \|_{L^\infty} \leq C \| f \|_{L^1},
	\end{align*}
	for a fixed $\eta > 0$, using the compact support of $\chi_+'$ to absorb the exponential factor, and then using Lemma \ref{l: right resolvent exponentially localized boundedness} to estimate the remaining term involving $\psi^+$. By Corollary \ref{c: psi resolvent farfield core invertibility} and Lemma \ref{l: psi resolvent tilde f estimates}, we have 
	\begin{align*}
		\| \partial_x T(\gamma) \tilde{f}(\gamma) \|_{L^1} \leq C \| f \|_{L^1}. 
	\end{align*}
	Finally, using similar arguments to the estimates on $\partial_x (\chi_+ \psi^+)$, we obtain 
	\begin{align*}
		\| B_+(\gamma) \tilde{f}(\gamma) \partial_x (\chi_+ e^{-\gamma \cdot} )\|_{L^1}  \leq \frac{C}{|\gamma|} \|f\|
	\end{align*}
	for $\gamma$ small with $\Re \gamma \geq \frac{1}{2} | \Im \gamma|$. Hence $\partial_x \psi^\mathrm{right}$ satisfies the desired estimate. 
\end{proof}

The regularity in $\gamma$ of spatial derivatives is further improved when we measure in $L^\infty$ instead of $L^1$. 

\begin{prop}\label{p: psi resolvent derivative L1 to L inf estimate}
	There exist constants $C > 0$ and $\delta > 0$ such that for all $f \in L^1 (\R) \cap L^\infty(\R)$, the derivative of the solution to $(\mcl_\psi - \gamma^2)\psi = f$ admits a decomposition 
	\begin{align}
		\partial_x \psi = \Psi^\mathrm{left}_d + \Psi^\mathrm{right}_d
	\end{align}
	with
	\begin{align}
		\| \Psi^\mathrm{left}_d (\gamma) - \Psi^\mathrm{left}_d (0) \|_{L^\infty} \leq C |\gamma|^2 \left( \| f \|_{L^1} + \| f \|_{L^\infty} \right) 
	\end{align}
	for all $\gamma \in B(0,\delta)$ with $\gamma^2$ to the right of the essential spectrum of $\mcl_\psi^-$, and 
	\begin{align}
		\| \Psi^\mathrm{right}_d\|_{L^\infty} \leq C \left(  \| f \|_{L^1} + \| f \|_{L^\infty} \right) 
	\end{align}
	for all $\gamma \in B(0,\delta)$ with $\Re \gamma \geq \frac{1}{2} | \Im \gamma |$. 
\end{prop}
\begin{proof}
	Using the far-field/core decomposition to solve for $\psi$, we again write. 
	\begin{align*}
		\partial_x \psi = \partial_x (\chi_- \psi^-) + \partial_x (\chi_+ \psi^+) + B_- (\gamma) \tilde{f}(\gamma) \partial_x (\chi_- e^{\nu^+(\gamma^2) \cdot}) + \partial_x T(\gamma) \tilde{f}(\gamma) + B_+(\gamma) \tilde{f}(\gamma) \partial_x (\chi_+ e^{-\gamma \cdot}). 
	\end{align*}
	We then define 
	\begin{multline}
		\Psi^\mathrm{right}_d = \chi_-' \psi^- + \partial_x(\chi_+ \psi^+) + B_- (\gamma) \tilde{f}(\gamma) \chi_-' e^{\nu^+(\gamma^2)\cdot} + \partial_x T(\gamma) \tilde{f}(\gamma)  \\ + B_+(\gamma) \tilde{f}(\gamma) \partial_x (\chi_+ e^{-\gamma \cdot}). 
	\end{multline}
	We group the terms $\chi_-' \psi^-$ and $B_- (\gamma) \tilde{f}(\gamma) \chi_-' e^{\nu^+(\gamma^2)\cdot}$ in $\Psi^\mathrm{right}_d$ even though they involve terms originating from the neutral mode on the left because they do not have sufficient regularity in $\gamma$ to obtain decay at the optimal rate when integrating along a parabolic contour tangent to the imaginary axis. Indeed, $\tilde{f}(\gamma)$ is only Lipschitz in $\gamma$, whereas to obtain decay at with rate $t^{-1}$ by integrating over such a contour, one would need expansions to order $\lambda = \gamma^2$. However, these terms have a compactly supported factor $\chi_-'$ which can absorb the small exponential growth of $e^{\nu^+(\gamma^2)\cdot}$ when $\gamma^2$ passes to the left of the essential spectrum of $\mcl_\psi^-$ but remains small, and hence we can still control these terms in this region and thereby use the contours otherwise reserved for terms originating from the right. 
	
	More precisely, by Lemma \ref{l: psi left resolvent exponential localization}, we have 
	\begin{align*}
		\| \chi_-' \psi^- (\gamma) \|_{L^\infty} \leq C \| f \|_{L^1}
	\end{align*}
	for any $\gamma \in B(0,\delta)$ with $\Re \gamma \geq 0$ for $\delta$ sufficiently small. By a similar argument, also using Corollary \ref{c: psi resolvent farfield core invertibility} and Lemma \ref{l: psi resolvent tilde f estimates} we obtain 
	\begin{align*}
		 \|B_- (\gamma) \tilde{f}(\gamma) \chi_-' e^{\nu^+(\gamma^2)\cdot}\|_{L^\infty} \leq C \| f \|_{L^1\cap L^\infty}
	\end{align*}
	for all $\gamma \in B(0,\delta)$ with $\Re \gamma \geq 0$. The other terms in $\Psi^\mathrm{right}_d$ are bounded in $L^\infty$ for $\gamma \in B(0,\delta)$ with $\Re \gamma \geq \frac{1}{2} | \Im \gamma |$ by arguments similar to those in the proof of the preceding proposition. 
	
	The remaining terms collected in $\Psi^\mathrm{left}_d$ are 
	\begin{align}
		\Psi^\mathrm{left}_d = \chi_- \partial_x \psi^- + B_-(\gamma) \tilde{f}(\gamma) \chi_- \nu^+(\gamma^2) e^{\nu^+(\gamma^2) \cdot}. 
	\end{align}
	The first term satisfies the desired expansion by Lemma \ref{l: psi left resolvent derivative Linf estimate}. For the second term, we have by Corollary \ref{c: psi resolvent farfield core invertibility} and Lemma \ref{l: psi resolvent tilde f estimates}
	\begin{align*}
		\| B_-(\gamma) \tilde{f}(\gamma) \chi_- \nu^+(\gamma^2) e^{\nu^+(\gamma^2) \cdot} \|_{L^\infty} \leq C |\nu^+(\gamma^2)| \| f \|_{L^1 \cap L^\infty} \| e^{\nu^+(\gamma^2)\cdot} \|_{L^\infty} \leq C |\gamma|^2 \| f \|_{L^1 \cap L^\infty}
	\end{align*}
	for $\gamma$ small with $\gamma^2$ to the right of the essential spectrum of $\mcl_\psi^-$, which completes the proof of the proposition. 
\end{proof}

\section{Linear estimates}\label{s: linear estimates}

We now translate the regularity of the resolvents of $\mcl_p$ and $\mcl_\psi$ near their essential spectra  into time decay estimates on the semigroups they generate. As elliptic operators with smooth bounded coefficients, $\mcl_p$ and $\mcl_\psi$ are both sectorial operators on $L^p(\R)$ for any $1 \leq p \leq \infty$, and hence they generate analytic semigroups through the inverse Laplace transform,

\begin{align}
	e^{\mcl_p t} &= - \frac{1}{2 \pi i} \int_{\Gamma_p} e^{\lambda t} (\mcl_p - \lambda)^{-1} \, d \lambda, \label{e: p semigroup formula} \\
	e^{\mcl_\psi t} &= - \frac{1}{2 \pi i} \int_{\Gamma_\psi} e^{\lambda t} (\mcl_\psi - \lambda)^{-1} \, d \lambda,
\end{align}
for appropriately chosen contours $\Gamma_p, \Gamma_\psi$. We take advantage of the estimates on the resolvents obtained in Sections \ref{s: p resolvent estimates} and \ref{s: psi resolvent estimates} to deform the integration contours near the essential spectrum and thereby extract temporal decay. 

\subsection{Large time estimates on $e^{\mcl_p t}$}

First we state the $t^{-3/2}$ decay estimate on $e^{\mcl_p t}$, which follows as in \cite{AveryScheel, AveryScheelSelection} from the regularity of the resolvent in Proposition \ref{p: p resolvent lipschitz}.  

\begin{figure}
	\centering
	\includegraphics[width = 1\textwidth]{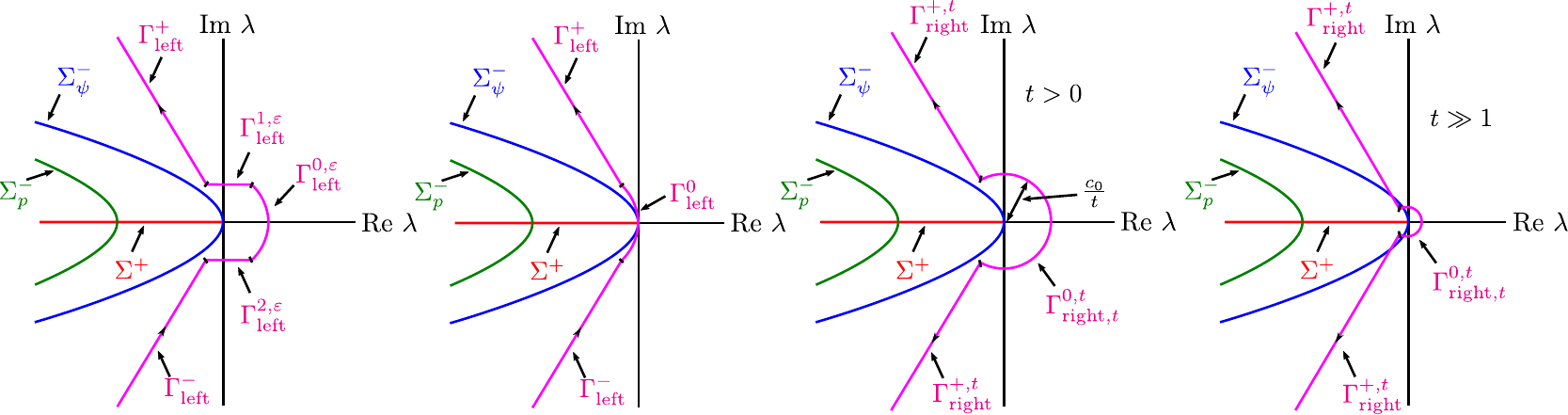}
	\caption{\footnotesize{Integration contours $\Gamma_\mathrm{left}^\eps$ for $\eps > 0$ and $\eps = 0$ (left, center left) and $\Gamma_\mathrm{right}^t$ for $t > 0$ and $t \gg 1$ (center right, right).} }
	\label{f: contours}
\end{figure}

\begin{prop}[\cite{AveryScheelSelection}, Proposition 4.2]\label{p: linear p t-3/2 estimate}
	There exists a constant $C > 0$ such that for any $p_0 \in L^1_{0, 1} (\R)$, we have for all $t > 0$ 
	\begin{align}
		\| e^{\mcl_p t} p_0 \|_{W^{1,\infty}_{0, -1}} \leq \frac{C}{t^{3/2}} \| p_0 \|_{L^1_{0,1}}. 
	\end{align}
\end{prop}

In our nonlinear argument, we will need to estimate $p(t)$ in the stronger norm $L^1$, and so we also state and prove the following linear decay estimate in this space. 

\begin{prop}\label{p: p L11 to L1 estimate}
	There exists a constant $C > 0$ such that for any $p_0 \in L^1_{0, 1} (\R)$, we have for all $t > 0$ 
	\begin{align}
		\| e^{\mcl_p t} p_0 \|_{L^1} \leq \frac{C}{t^{1/2}} \| p_0 \|_{L^1_{0, 1}}. 
	\end{align}
\end{prop}
\begin{proof}
	By Proposition \ref{p: p resolvent L11 to L1 estimate}, for any $p_0 \in L^1_{0,1} (\R)$, we have 
	\begin{align*}
		\| (\mcl_p - \gamma^2)^{-1} p_0 \|_{L^1} \leq \frac{C}{ |\gamma|} \| p_0 \|_{L^1_{0,1}} 
	\end{align*}
	for any $\gamma$ sufficiently small with $\Re \gamma \geq \frac{1}{2} | \Im \gamma|$. With this estimate in hand, we obtain time decay from a standard argument, choosing $\Gamma_p$ in \eqref{e: p semigroup formula} to be a circular arc centered at the origin whose radius scales as $t^{-1}$, connected to two rays extending out to infinity in the left half plane; see the right two panels in Figure \ref{f: contours}. The above estimate on the rate of blow up of the resolvent then translates into the desired decay rate; see, for instance, \cite[Proposition 7.4]{AveryScheel}. 
\end{proof}


\subsection{Large time estimates on $e^{\mcl_\psi t}$}

The first linear estimate we prove here states that $\psi$ decays diffusively, with the expected $L^1$-$L^\infty$ estimate. 

\begin{prop}\label{p: psi L1 Linf estimate}
	There exists a constant $C > 0$ such that for any $\psi_0 \in L^1 (\R)$, we have for all $t > 0$
	\begin{align}
		\| e^{\mcl_\psi t} \psi_0 \|_{L^\infty} \leq \frac{C}{t^{1/2}} \| \psi_0 \|_{L^1}. 
	\end{align}
\end{prop}
\begin{proof}
	First we restrict to $\psi_0 \in L^1 (\R) \cap L^\infty(\R)$. Decomposing $(\mcl_\psi-\gamma^2)^{-1} \psi_0 = \psi^\mathrm{left} (\gamma) + \psi^\mathrm{right}(\gamma)$ as in Proposition \ref{p: psi resolvent L1 to L inf estimate}, we write 
	\begin{align*}
		e^{\mcl_\psi t} \psi_0 = -\frac{1}{2 \pi i} \int_{\Gamma_\psi} e^{\gamma^2 t} \psi^\mathrm{left} (\gamma) \, d (\gamma^2) - \frac{1}{2 \pi i} \int_{\Gamma_\psi} e^{\gamma^2 t} \psi^\mathrm{right}(\gamma) \, d (\gamma^2). 
	\end{align*}
	The terms $\psi^\mathrm{right} (\gamma)$ and $\psi^\mathrm{left}(\gamma)$ remain analytic in $\gamma^2$ for $\gamma^2$ in the resolvent set of $\mcl_\psi$, so we may separately deform the contours in these two integrals by Cauchy's theorem. For the integral involving $\psi^\mathrm{left}(\gamma)$, we first fix $\eps > 0$ small and integrate along the contour
	\begin{align*}
		\Gamma_\mathrm{left}^\eps = \Gamma^+_\mathrm{left} \cup \Gamma^{1,\eps}_\mathrm{left} \cup \Gamma^{0, \eps}_\mathrm{left} \cup^{2, \eps}_\mathrm{left} \cup \Gamma^-_\mathrm{left} 
	\end{align*}
	pictured in the left most panel in Figure \ref{f: contours}. We choose the lengths of the line segments $\Gamma^{1,\eps}_\mathrm{left}, \Gamma^{2, \eps}_\mathrm{left}$ to be equal to $\eps$. The boundedness of the resolvent from Proposition \ref{p: psi resolvent L1 to L inf estimate} then guarantees that 
	\begin{align*}
		\lim_{\eps \to 0^+} \left\| \int_{\Gamma^{j, \eps}_\mathrm{left}} e^{\lambda t} (\mcl_\psi - \lambda)^{-1} \, d \lambda \right\|_{L^1 \to L^\infty} = 0. 
	\end{align*}
	Hence we may send $\eps \to 0^+$ and thereby write the semigroup as the integral over the limiting contour $\Gamma_\mathrm{left} = \Gamma^+_\mathrm{left} \cup \Gamma^0_\mathrm{left} \cup \Gamma^-_\mathrm{left}$ pictured in the center left panel of Figure \ref{f: contours}. We parameterize the parabolic arc tangent to the imaginary axis as 
	\begin{align*}
		\Gamma_\mathrm{left}^0 = \{ i a - c a^2 : a \in [-\delta, \delta]\}
	\end{align*}
	for some $\delta, c > 0$ sufficiently small. The contribution from the rays extending out to infinity is exponentially decaying in time, since these are contained strictly in the left half plane, so we focus on estimating the integral over $\Gamma_{\mathrm{left}}^0$, for which we have 
	\begin{align*}
		\int_{\Gamma_\mathrm{left}^0} e^{\gamma^2 t} \psi^\mathrm{left} (\gamma) \, d (\gamma^2) = \int_{-\delta}^\delta e^{i a t - c a^2 t} \psi^\mathrm{left} (\sqrt{ia - ca^2}) (i - 2 ca ) \, da. 
	\end{align*}
	By the control of $\psi^\mathrm{left}$ in Proposition \ref{p: psi resolvent L1 to L inf estimate}, we then have 
	\begin{align*}
		\left\| \int_{\Gamma_\mathrm{left}^0} e^{\gamma^2 t} \psi^\mathrm{left} (\gamma) \, d (\gamma^2) \right\|_{L^\infty} &\leq C \| \psi_0 \|_{L^1 \cap L^\infty} \int_{-\delta}^\delta e^{-ca^2 t} \, da \\
		&= C \| \psi_0 \|_{L^1 \cap L^\infty} \frac{1}{\sqrt{t}} \int_{-\delta \sqrt{t}}^{\delta \sqrt{t}} e^{-cz^2} \, dz \\
		&\leq \frac{C}{\sqrt{t}} \| \psi_0 \|_{L^1 \cap L^\infty}. 
	\end{align*}

	For the term involving $\psi^\mathrm{right}(\gamma)$, we integrate over the same contour used in the proof of Proposition \ref{p: p L11 to L1 estimate}, pictured in the right two panels of Figure \ref{f: contours}. Note that for $t$ large, the contour $\Gamma_\mathrm{right}^t$ passes through the essential spectrum of $\mcl_\psi$; however, the resolvent is still pointwise analytic in $\gamma$ in this region by Proposition \ref{p: psi resolvent L1 to L inf estimate}, so that the resulting contour integral is well defined and still determines the action of the semigroup $e^{\mcl_\psi t}$ \cite{ZumbrunHoward}. We thereby obtain, using the estimate on $\psi^\mathrm{right}$ from Proposition \ref{p: psi resolvent L1 to L inf estimate},
	\begin{align*}
		\left\| \frac{1}{2 \pi i} \int_{\Gamma_\psi} e^{\gamma^2 t} \psi^\mathrm{right}(\gamma) \, d (\gamma^2) \right\|_{L^\infty} \leq \frac{C}{\sqrt{t}} \| \psi_0\|_{L^1 \cap L^\infty}.
	\end{align*}
	The restriction to $\psi_0 \in L^1 (\R) \cap L^\infty (\R)$ may be removed by combining these estimates with the standard small time parabolic regularity estimate 
	\begin{align*}
		\| e^{\mcl_\psi t} \psi_0 \|_{L^\infty} \leq \frac{C}{\sqrt{t}} \| \psi_0 \|_{L^1},
	\end{align*}
	for $0 < t < 1$; see Lemma \ref{l: small time estimates} below. Hence we obtain 
	\begin{align*}
		\| e^{\mcl_\psi t} \|_{L^\infty} \leq \frac{C}{\sqrt{t}} \| \psi_0 \|_{L^1}
	\end{align*}
	for all $t  > 0$, as desired. 
\end{proof}

In the nonlinear argument, we will need to establish sufficiently fast decay of terms of the form $f \psi(t)^2$, where $f$ is a function which is exponentially localized on the left. The pure diffusive decay rate $t^{-1/2}$ of $\psi$ is not sufficient to close our nonlinear argument for such a term. However, we can take advantage of the outward transport as well as the localization of $f$ to obtain a stronger decay estimate, sufficient to close our nonlinear argument, by measuring in a norm which allows some algebraic growth on the left. 

\begin{prop}\label{p: psi L1 01 Linf-10 estimate}
	There exists a constant $C > 0$ such that for any $\psi_0 \in L^1_{0, 1} (\R)$, we have for all $t > 0$
	\begin{align}
		\| e^{\mcl_\psi t} \psi_0 \|_{L^\infty_{-1, 0}} \leq \frac{C}{t} \| \psi_0 \|_{L^1_{0, 1}}.
	\end{align} 
\end{prop}
\begin{proof}
	We again restrict at first to $\psi_0 \in L^1_{0,1} (\R) \cap L^\infty(\R)$ and use Proposition \ref{p: psi resolvent L1 to Linf-10} to decompose the resolvent, this time writing 
	\begin{multline}
		e^{\mcl_\psi t} \psi_0 = -\frac{1}{2 \pi i } \int_{\Gamma_\psi} e^{\gamma^2 t} \tilde{\psi}^\mathrm{left}(0) \, d (\gamma^2) - \frac{1}{2 \pi i} \int_{\Gamma_\psi} e^{\gamma^2 t} (\tilde{\psi}^\mathrm{left}(\gamma) - \tilde{\psi}^\mathrm{left}(0)) \, d (\gamma^2) \\ -\int_{\Gamma_\psi} \frac{1}{2 \pi i} e^{\gamma^2 t} \tilde{\psi}^\mathrm{right} (\gamma) \, d (\gamma^2). 
	\end{multline}
	The integrand in the first integral only depends on $\gamma$ through $e^{\gamma^2 t}$, and hence is analytic in $\gamma^2$ on $\C$, so that this term vanishes. For the second integral, we integrate over the contour $\Gamma_\mathrm{left}$ used in the proof of Proposition \ref{p: psi L1 Linf estimate} and use Proposition \ref{p: psi resolvent L1 to Linf-10} to estimate the remainder $\tilde{\psi}^\mathrm{left}(\gamma) - \tilde{\psi}^\mathrm{left}(0)$ to obtain
	\begin{align*}
		\left\| \frac{1}{2 \pi i} \int_{\Gamma^0_\mathrm{left}} e^{\gamma^2 t} (\tilde{\psi}^\mathrm{left}(\gamma) - \tilde{\psi}^\mathrm{left}(0)) \, d (\gamma^2) \right\|_{L^\infty_{-1, 0}} &\leq C\| \psi_0 \|_{L^1 \cap L^\infty} \int_{-\delta}^{\delta} e^{-c a^2 t} | i a - ca^2 | |i - 2 ca| \, da \\
		&\leq \frac{C}{t} \| \psi_0 \|_{L^1\cap L^\infty},
	\end{align*}
	gaining an extra factor of $t^{-1/2}$ in decay from the $\mathrm{O}(|\gamma|^2)$ estimate on the remainder. The contributions from the other parts of the contour $\Gamma_\mathrm{left}$ are again exponentially decaying in time.  
	
	The estimate on the term involving $\tilde{\psi}^\mathrm{right}(\gamma)$ follows as in the proof of Proposition \ref{p: psi L1 Linf estimate}, but we gain an extra factor of $t^{-1/2}$ since $\| \tilde{\psi}^\mathrm{right}(\gamma)\|_{L^\infty}$ is bounded rather than blowing up at rate $|\gamma|^{-1}$ by Proposition \ref{p: psi resolvent L1 to Linf-10}. The restriction to $\psi_0 \in L^\infty(\R)$ can again be removed by the small time regularity estimate in Lemma \ref{l: small time estimates}. 
\end{proof}

Finally, we establish the following decay estimates on derivatives, equivalent to those for the heat equation.

\begin{prop}\label{p: linear psi derivative estimates}
	There exists a constant $C > 0$ such that for any $\psi_0 \in L^1 (\R)$, we have for all $t > 0$
	\begin{align}
		\| \partial_x (e^{\mcl_\psi t} \psi_0) \|_{L^1} &\leq \frac{C}{\sqrt{t}} \| \psi_0 \|_{L^1}, \label{e: psi derivative L1 to L1 estimate}
	\end{align}
	and 
	\begin{align}
		\| \partial_x (e^{\mcl_\psi t} \psi_0) \|_{L^\infty} &\leq \frac{C}{t} \| \psi_0 \|_{L^1}. \label{e: psi derivative L1 to Linf estimate} 
	\end{align}
\end{prop}
\begin{proof}
	The proof of \eqref{e: psi derivative L1 to L1 estimate} is completely analogous to that of Proposition \ref{p: psi L1 Linf estimate}, with Proposition \ref{p: psi resolvent derivative L1 to L1 estimate} replacing Proposition \ref{p: psi resolvent L1 to L inf estimate}. The proof of \eqref{e: psi derivative L1 to Linf estimate} is similar, using Proposition \ref{p: psi resolvent derivative L1 to L inf estimate} but gaining an extra factor of $t^{-1/2}$ decay due to the $\mathrm{O}(|\gamma|^2)$ estimate on the remainder $\| \Psi^\mathrm{left}_d(\gamma) - \Psi^\mathrm{left}_d (0)\|_{L^\infty}$. 
\end{proof}

\subsection{Small time estimates}
As $\mcl_p$ and $\mcl_\psi$ are both elliptic operators with smooth, bounded coefficients, $e^{\mcl_p t}$ and $e^{\mcl_\psi t}$ obey the same well known small time regularity estimates; see for instance \cite{Lunardi, Henry}. The estimates in this section thereby hold for $\mcl = \mcl_p$ or $\mcl_\psi$. 

\begin{lemma}\label{l: small time estimates}
	Fix $1 \leq \ell \leq \infty$. There exists a constant $C > 0$ such that for all $0 < t < 1$, we have 
	\begin{align}
		\| e^{\mcl t} f \|_{L^\infty} &\leq \frac{C}{t^{1/2}} \| f \|_{L^1} \text{ for all } f \in L^1 (\R), \label{e: small time L1 L inf estimate} \\
		\| e^{\mcl t} f \|_{W^{1, \ell}} &\leq C \| f \|_{W^{1,\ell}}  \text{ for all } f \in W^{1, \ell} (\R), \label{e: small time W 1 ell estimate}
	\end{align}
	and 
	\begin{align}
		\| \partial_{xx} (e^{\mcl t} f) \|_{L^\infty} \leq \frac{C}{t^{1/2}} \| f \|_{W^{1,\infty}}  \text{ for all } f \in W^{1, \infty} (\R). \label{e: small time 2nd derivative estimate}
	\end{align}
\end{lemma}

%
%

\section{Nonlinear stability --- proof of Theorem \ref{t: main}}\label{s: nonlinear stability}
The system \eqref{e: p equation}-\eqref{e: psi equation} for $(p, \psi)$ is locally well posed in $L^1 (\R) \cap L^\infty(\R) \times W^{1, \infty} (\R)$. By parabolic regularity, a solution with initial data $(p_0, \psi_0)$ small in $L^1 (\R) \cap L^\infty(\R) \times W^{1, \infty} (\R)$ is small in $W^{1, \infty} (\R) \cap W^{1,1} (\R) \times W^{1, \infty} (\R)$ after a fixed short time. Hence in proving Theorem \ref{t: main}, we may assume without loss of generality that $\partial_x p_0$ is small in $L^1 \cap L^\infty$.

The system \eqref{e: p equation}-\eqref{e: psi equation} for $(p, \psi)$ is also locally well posed (in particular) in $W^{1, \infty}_{0, -1} (\R) \times W^{1, \infty} (\R)$ by standard theory of semilinear parabolic equations \cite{Henry, Lunardi}, in the sense that given any $(p_0, \psi_0) \in W^{1, \infty}_{0, -1} (\R) \times W^{1, \infty} (\R)$, there exists a maximal existence time $T_* \in (0, \infty]$ and a unique solution $(p(t), \psi(t))$ to \eqref{e: p equation}-\eqref{e: psi equation} for $t \in (0, T_*)$ satisfying 
\begin{align}
	\lim_{t \to 0^+} \begin{pmatrix}
		(\mcl_p - \lambda)^{-1} & 0 \\
		0 & (\mcl_\psi - \lambda)^{-1}
	\end{pmatrix} 
	\begin{pmatrix}
		p(t) \\
		\psi(t) 
	\end{pmatrix} = 
	\begin{pmatrix}
		(\mcl_p - \lambda)^{-1} & 0 \\
		0 & (\mcl_\psi - \lambda)^{-1}
	\end{pmatrix} 
	\begin{pmatrix}
		p_0 \\
		 \psi_0
	\end{pmatrix}
\end{align}
for any $\lambda$ in the resolvent sets of both $\mcl_p$ and $\mcl_\psi$. Furthermore, the maximal existence time $T_*$ depends only on $\| (p_0, \psi_0) \|_{W^{1, \infty}_{0, -1} \times W^{1, \infty}}$. For the remainder of this section, we let $(p, \psi)$ be a solution with initial data $(p_0, \psi_0) \in L^1_{0,1} (\R) \cap W^{1,\infty} (\R) \cap W^{1,1} (\R) \times L^1_{0,1} (\R) \cap W^{1, \infty}(\R)$ and let $T_*$ denote its maximal existence time in $W^{1, \infty}_{0, -1} (\R) \times W^{1,\infty} (\R)$. 

We then define the norm
\begin{multline}
	\Theta(t) = \sup_{0 < s \leq t} \bigg[ (1+s)^{3/2} \| p(s) \|_{W^{1, \infty}_{0, -1}} + (1+s)^{1/2} \| p(s) \|_{W^{1,1}} + (1+s)^{1/2} \| \psi(s) \|_{L^\infty} \\ + (1+s) \| \psi(s) \|_{L^\infty_{-1, 0}}  + s^{1/2} \| \psi_x (s) \|_{L^1} + (1+s) \| \psi_x (s) \|_{L^\infty}+ 1_{\{0 < s \leq 1\}} s^{1/2} \|\psi_{xx}(s)\|_{L^\infty} \bigg]. \label{e: theta def}
\end{multline}

By the local well-posedness theory, uniform control of $\Theta(t)$ up to time $T_*$ implies global existence of the solution to \eqref{e: p equation}-\eqref{e: psi equation} together with the desired decay estimates. 

To control the nonlinearity, we will use the fact that by Taylor's theorem, there exists a constant $C > 0$ such that 
\begin{align}
	\left| \frac{1}{1 + \omega^{-1} q_*^{-1} p} - 1 \right| \leq C | \omega^{-1} q_*^{-1} p|, \quad \left| \frac{1}{1 + \omega^{-1} q_*^{-1} p} \right| \leq C, \label{e: nonlinearity Taylor expansion}
\end{align}
and 
\begin{align}
	| \omega^{-2} p^3| \leq C | \omega^{-1} p^2|,
\end{align}
provided $\| \omega^{-1} q_*^{-1} p \|_{L^\infty} \leq \frac{1}{2}$. We will obtain global control of $\Theta(t)$ through the estimate in the following proposition.  

\begin{prop}\label{p: theta control}
	There exist positive constants $C_1$ and $C_2$ such that the function $\Theta(t)$ from \eqref{e: theta def} satisfies
	\begin{align}
		\Theta(t) \leq C_1 \left( \| p_0\|_{L^1_{0, 1}} + \| p_0\|_{W^{1,\infty} \cap W^{1,1}} + \| \psi_0 \|_{L^1_{0,1}} + \| \psi_0 \|_{W^{1,\infty}} \right) + C_2 \Theta(t)^2,
	\end{align}
	for all $t \in (0, T_*)$, provided $\| \omega^{-1} q_*^{-1} p(s) \|_{L^\infty} \leq \frac{1}{2}$ for all $s \in (0, t)$. 
\end{prop}

We start by rewriting the system \eqref{e: p equation}-\eqref{e: psi equation} in mild form via the variation of constants formula, obtaining
\begin{align}
	p(t) &= e^{\mcl_p t} p_0 + \mathcal{I}_{p,1} (t) + \mathcal{I}_{p, 2} (t) + \mathcal{I}_{p, 3} (t) + \mathcal{I}_{p, 4} (t), \label{e: voc p} \\
	\psi(t) &= e^{\mcl_\psi t} \psi_0 + \mathcal{I}_{\psi, 1} (t) + \mathcal{I}_{\psi, 2} (t) + \mathcal{I}_{\psi, 3} (t) \label{e: voc psi},
\end{align} 
where
\begin{align}
	\mathcal{I}_{p, 1} (t) &= - \int_0^t e^{\mcl_p (t-s)} \left[ 3 q_* \omega^{-1} p(s)^2 + \omega^{-2} p(s)^3 \right] \, ds, \\
	\mathcal{I}_{p, 2} (t) &= - \int_0^t e^{\mcl_p (t-s)} \left[ \omega^{-1} q_*^{-1} \psi_x(s)^2 \right] \, ds, \\
	\mathcal{I}_{p, 3} (t) &= - \int_0^t e^{\mcl_p (t-s)} \left[ \omega q_* [(\omega^{-1} q_*^{-1})_x]^2 \psi(s)^2 \right] \, ds, \\
	\mathcal{I}_{p, 4} (t) &= - \int_0^t e^{\mcl_p (t-s)} \left[ p(s) [(\omega^{-1} q_*^{-1} \psi(s))_x]^2 \right] \, ds,
\end{align}
and 
\begin{align}
	\mathcal{I}_{\psi, 1} (t) &= \int_0^t e^{\mcl_\psi (t-s)} \left[ 2 \frac{q_*'}{q_*} \left( \frac{1}{1+ \omega^{-1} q_*^{-1} p} - 1 \right) \psi_x (s) \right] \, ds, \\
	\mathcal{I}_{\psi, 2} (t) &= \int_0^t e^{\mcl_\psi (t-s)} \left[ 2 \omega q_*' \left( \frac{1}{1+ \omega^{-1} q_*^{-1} p} - 1 \right) (\omega^{-1} q_*^{-1})_x \psi(s) \right] \, ds, \\
	\mathcal{I}_{\psi, 3} (t) &= \int_0^t e^{\mcl_\psi (t-s)} \left[ 2 \omega (\omega^{-1} p)_x \left( \frac{1}{1+ \omega^{-1} q_*^{-1}} \right) (\omega^{-1} q_*^{-1} \psi(s))_x \right] \, ds. 
\end{align}

The coefficients in the nonlinear terms gain spatial localization due to the factors involving $\omega$ and $q_*$. Precise estimates on this localization effect may be inferred from the front asymptotics \cite{aronson},
\begin{align}
	q_*(x) = \begin{cases}
		(a+bx) e^{-x} + \mathrm{O}(x^2 e^{-2x}), &x \to \infty, \\
		1 - c_1 e^{(-1 + \sqrt{2}) x} + \mathrm{O} (e^{(-2+2\sqrt{2})x}), & x \to -\infty.  
	\end{cases}
	\label{e: front asymptotics}
\end{align}

To characterize decay rates of nonlinearities throughout this section, we will make frequent use of the following elementary estimate.

\begin{lemma}\label{l: nonlinearity integral estimate}
	Fix $\alpha > 0$ and let $\beta > 1$, further assuming $\beta \geq \alpha$. There exists a constant $C > 0$ such that for all $t > 1$, we have 
	\begin{align}
		\int_0^{t-1} \frac{1}{(t-s)^\alpha} \frac{1}{(1+s)^\beta} \, ds \leq \frac{C}{t^\alpha}. 
	\end{align}
\end{lemma}
\begin{proof}
	The integral in question is clearly uniformly bounded for $1 < t < 2$, since $s$ is bounded away from $t$ in the region of integration, so there is no singularity in the integrand. We may therefore assume $t > 2$. When estimating the full nonlinearities, there will of course be a piece of the integral from $t-1$ to $t$, but here we focus on only the part from $0$ to $t-1$ in order to separate the large time decay from small time regularity estimates. 
	
	For $t > 2$, we note that $t-1 > \frac{t}{2}$, and write 
	\begin{align*}
		\int_0^{t-1} \frac{1}{(t-s)^\alpha} \frac{1}{(1+s)^\beta} \, ds = \int_0^{t/2} \frac{1}{(t-s)^\alpha} \frac{1}{(1+s)^\beta} \, ds + \int_{t/2}^{t-1} \frac{1}{(t-s)^\alpha} \frac{1}{(1+s)^\beta} \, ds. 
	\end{align*}
	In the first region of integration, $t-s \sim t$, and so we have 
	\begin{align*}
		\int_0^{t/2} \frac{1}{(t-s)^\alpha} \frac{1}{(1+s)^\beta} \, ds \leq \frac{C}{t^\alpha} \int_0^{t/2} \frac{1}{(1+s)^\beta} \, ds \leq \frac{C}{t^\alpha} \int_0^\infty \frac{1}{(1+s)^\beta} \, ds \leq \frac{C}{t^\alpha} 
	\end{align*}
	since $\beta > 1$. In the second region of integration $s \sim t$, and so we have 
	\begin{align*}
		\int_{t/2}^{t-1} \frac{1}{(t-s)^\alpha} \frac{1}{(1+s)^\beta} \, ds \leq \frac{C}{(1+t)^\beta} \int_{t/2}^{t-1} \frac{1}{(t-s)^\alpha} \, ds &= \frac{C}{(1+t)^\beta} \int_1^{t/2} \frac{1}{\tau^\alpha} \, d \tau 
	\end{align*}
	If $\alpha > 1$, we have 
	\begin{align*}
		\frac{C}{(1+t)^\beta} \int_1^{t/2} \frac{1}{\tau^\alpha} \, d \tau \leq \frac{C}{(1+t)^\beta} \leq \frac{C}{t^\alpha}
	\end{align*}
	for $t > 2$, since $\beta \geq \alpha$. If $\alpha = 1$, we gain a factor of $\log t$, but since $\beta > 1$, this can be absorbed by $(1+t)^{-\beta}$ while still retaining decay at rate $t^{-1}$. Finally, if $\alpha < 1$ we have 
	\begin{align*}
		\frac{C}{(1+t)^\beta} \int_1^{t/2} \frac{1}{\tau^\alpha} \, d \tau \leq \frac{C}{(1+t)^\beta} t^{1-\alpha} \leq \frac{C}{(1+t)^{\beta + \alpha - 1}} \leq \frac{C}{(1+t)^\alpha}
	\end{align*}
	for $t > 2$, since $\beta > 1$, so $\beta + \alpha - 1 > \alpha$, completing the proof of the lemma.
\end{proof}
We note that if $\alpha < 1$, then the integral from $0$ to $t-1$ in Lemma \ref{l: nonlinearity integral estimate} may be replaced by an integral from $0$ to $t$, since the singularity in the integrand near $s = t$ is integrable in this case.

We prove Proposition \ref{p: theta control} by breaking the estimate into several pieces, first establishing control of $p(t)$. 
\subsection{Estimates on $p(t)$}

We begin by estimating $\| p(t) \|_{W^{1,\infty}_{0, -1}}$, taking advantage of the estimates on $p$ and $\psi$ encoded in the definition of $\Theta$. 

\begin{lemma}[$W^{1,\infty}_{0,-1}$ estimates on $\mathcal{I}_{p, 1} (t)$]\label{l: Ip1 L inf 0 -1}
	There exists a constant $C > 0$ such that for all $t \in (0, T_*)$ we have 
	\begin{align*}
		(1+t)^{3/2} \| \mathcal{I}_{p,1} (t) \|_{W^{1, \infty}_{0,-1}} \leq C \Theta(t)^2
	\end{align*}
	provided $\| \omega^{-1} q_*^{-1} p(s) \|_{L^\infty} \leq \frac{1}{2}$ for all $s \in (0, t)$. 
\end{lemma}
\begin{proof}
	We split the integral in the definition of $\mathcal{I}_{p,1} (t)$ as 
	\begin{align*}
		\mathcal{I}_{p, 1} (t) = -\left( \int_0^{t-1} + \int_{t-1}^t \right) \left[ e^{\mcl_p (t-s)} \left[ 3 q_* \omega^{-1} p(s)^2 + \omega^{-2} p(s)^3 \right] \, ds \right]  
	\end{align*}
	For the first piece, we have by Proposition \ref{p: linear p t-3/2 estimate}
	\begin{multline*}
		(1+t)^{3/2} \left\| \int_0^{t-1} e^{\mcl_p (t-s)} \left[ 3 q_* \omega^{-1} p(s)^2 + \omega^{-2} p(s)^3 \right] \, ds \right\|_{W^{1,\infty}_{0,-1}} \\ \leq C (1+t)^{3/2} \int_0^{t-1} \frac{1}{(t-s)^{3/2}} \left\| 3 q_* \omega^{-1} p(s)^2 + \omega^{-2} p(s)^3 \right \|_{L^1_{0,1}} \, ds.
	\end{multline*}
	The nonlinear terms have exponentially localized coefficients, so that 
	\begin{align*}
		\left\| 3 q_* \omega^{-1} p(s)^2 + \omega^{-2} p(s)^3 \right \|_{L^1_{0,1}} &\leq C \left( \| p(s) \|_{L^1} \| p(s) \|_{L^\infty_{0,-1}} + C \| p(s) \|_{L^1} \| p(s) \|_{L^\infty_{0,-1}} \right) \\
		&\leq C \Theta(s)^2 \frac{1}{(1+s)^2},
	\end{align*}
	provided $\| \omega^{-1} q_*^{-1} p(s) \|_{L^\infty} \leq \frac{1}{2}$ for all $s \in (0, t)$.
	Hence, since $\Theta(t)$ is non-decreasing,  we have 
	\begin{multline*}
		(1+t)^{3/2} \left\| \int_0^{t-1} e^{\mcl_p (t-s)} \left[ 3 q_* \omega^{-1} p(s)^2 + \omega^{-2} p(s)^3 \right] \, ds \right\|_{W^{1,\infty}_{0,-1}} \\ \leq C \Theta(t)^2 (1+t)^{3/2} \int_0^{t-1} \frac{1}{(t-s)^{3/2}} \frac{1}{(1+s)^2} \, ds
	\end{multline*}
	By Lemma \ref{l: nonlinearity integral estimate}, we have 
	\begin{align*}
		(1+t)^{3/2} \int_0^{t-1} \frac{1}{(t-s)^{3/2}} \frac{1}{(1+s)^2} \, ds \leq C,
	\end{align*}
	so that we have the desired estimate for the piece of the integral from $0$ to $t-1$. The estimate on the integral from $t-1$ to $t$ is similar, but we use the small time regularity estimate \eqref{e: small time W 1 ell estimate} from Lemma \ref{l: small time estimates} in place of Proposition \ref{p: linear p t-3/2 estimate}. 
\end{proof}

\begin{lemma}[$W^{1, \infty}_{0, -1}$ estimates on $\mathcal{I}_{p, 2} (t)$]\label{l: Ip2 L inf 0 -1 estimates}
	There exists a constant $C > 0$ such that for all $t \in (0, T_*)$, we have 
	\begin{align}
		(1+t)^{3/2} \| \mathcal{I}_{p, 1} (t) \|_{W^{1, \infty}_{0, -1}} \leq C \Theta(t)^2. 
	\end{align}
\end{lemma}
\begin{proof}
	For $0 < t <1$, we have, using Lemma \ref{l: small time estimates},
	\begin{align*}
		(1+t)^{3/2} \| \mathcal{I}_{p, 2} \|_{W^{1, \infty}_{0, -1}} &\leq C (1+t)^{3/2} \int_0^t \left\| e^{\mcl_p (t-s)} \left[ \omega^{-1} q_*^{-1} \psi_x(s)^2 \right] \right\|_{W^{1,\infty}} \, ds \\
		&\leq C \int_0^t \| \omega^{-1} q_*^{-1} \psi_x (s)^2 \|_{W^{1, \infty}} \, ds \\
		&\leq C \int_0^t \left( \| \psi_x (s)\|_{L^\infty}^2 + \| \psi_{xx}(s) \|_{L^\infty} \| \psi_x(s) \|_{L^\infty} \right) \, ds \\
		&\leq C \Theta(t)^2 \int_0^t \left( \frac{1}{(1+s)^2} + \frac{s^{-1/2}}{1+s} \right) \, ds \\
		&\leq C \Theta(t)^2, 
	\end{align*}
	as desired, also using the fact that $\Theta(t)$ is non-decreasing by definition. 
	
	For $t > 1$, we split the integral into two pieces in order to handle small time regularity separately from the decay estimates for large times, writing
	\begin{align}
		\mathcal{I}_{p, 2} (t) = - \int_0^{t-1} e^{\mcl_p (t-s)} \left[ \omega^{-1} q_*^{-1} \psi_x (s)^2 \right] \, ds - \int_{t-1}^t e^{\mcl_p (t-s)} \left[ \omega^{-1} q_*^{-1} \psi_x (s)^2 \right] \, ds.
	\end{align}
	The integral from $t-1$ to $t$ is estimated as in the $0 < t < 1$ case above. For the other integral, we have, using Proposition \ref{p: linear p t-3/2 estimate},
	\begin{multline}
		(1+t)^{3/2} \left\| \int_0^{t-1} e^{\mcl_p (t-s)} \left[ \omega^{-1} q_*^{-1} \psi_x (s)^2 \right] \, ds \right\|_{W^{1,\infty}_{0, -1}} \\ \leq C (1+t)^{3/2} \int_0^{t-1} \frac{1}{(t-s)^{3/2}} \| \omega^{-1} q_*^{-1} \psi_x(s)^2 \|_{L^1_{0,1}} \, ds. 
	\end{multline}
	By the front asymptotics \eqref{e: front asymptotics}, we have $| \omega^{-1} q_*^{-1} | \leq C \rho_{0, -1}$, so that 
	\begin{align*}
		\| \omega^{-1} q_*^{-1} \psi_x(s)^2 \|_{L^1_{0,1}} \leq C \| \psi_x(s)^2 \|_{L^1} \leq C \| \psi_x (s) \|_{L^1} \| \psi_x(s) \|_{L^\infty} \leq C \Theta(s)^2 \frac{1}{(1+s)^{3/2}}. 
	\end{align*}
	Hence we obtain
	\begin{align*}
		(1+t)^{3/2} \left\| \int_0^{t-1} e^{\mcl_p (t-s)} \left[ \omega^{-1} q_*^{-1} \psi_x (s)^2 \right] \, ds \right\|_{W^{1,\infty}_{0, -1}} &\leq C (1+t)^{3/2} \Theta(t)^2 \int_0^{t-1} \frac{1}{(t-s)^{3/2}} \frac{1}{(1+s)^{3/2}} \, ds \\
		&\leq C \Theta(t)^2
	\end{align*}
	by Lemma \ref{l: nonlinearity integral estimate} and the fact that $\Theta(t)$ is non-decreasing, as desired. 
\end{proof}

\begin{lemma}[$W^{1, \infty}_{0, -1}$ estimates on $\mathcal{I}_{p, 3} (t)$]\label{l: Ip3 W 1 inf 0 -1 estimates}
	There exists a constant $C > 0$ such that for all $t \in (0, T_*)$, we have 
	\begin{align}
		(1+t)^{3/2} \| \mathcal{I}_{p, 3} (t) \|_{W^{1, \infty}_{0,-1}} \leq C \Theta(t)^2. 
	\end{align}
\end{lemma}
\begin{proof}
	We focus on the case $t > 1$, as the small time estimates can be handled in a similar manner to the proof of Lemma \ref{l: Ip2 L inf 0 -1 estimates}. We again split the integral into two pieces, writing 
	\begin{align}
		\mathcal{I}_{p, 3} (t) = - \int_0^{t-1} e^{\mcl_p (t-s)} \left[ \omega q_* [(\omega^{-1} q_*^{-1})_x ]^2 \psi(s)^2 \right] \, ds - \int_{t-1}^t  e^{\mcl_p (t-s)} \left[ \omega q_* [(\omega^{-1} q_*^{-1})_x ]^2 \psi(s)^2 \right] \, ds. 
	\end{align}
	Again, we focus on the integral from 0 to $t-1$, as the other integral involves only small time estimates similar to those in the proof of Lemma \ref{l: Ip2 L inf 0 -1 estimates}. For this first integral, we have 
	\begin{multline}
		(1+t)^{3/2} \left\| \int_0^{t-1}  e^{\mcl_p (t-s)} \left[ \omega q_* [(\omega^{-1} q_*^{-1})_x ]^2 \psi(s)^2 \right] \, ds \right\|_{W^{1, \infty}_{0, -1}} \\ \leq C (1+t)^{3/2} \int_0^{t-1} \frac{1}{(t-s)^{3/2}} \| \omega q_* [ (\omega^{-1} q_*^{-1})_x ]^2 \psi(s)^2 \|_{L^1_{0, 1}} \, ds.
	\end{multline}
	By the front asymptotics \eqref{e: front asymptotics}, we have 
	\begin{align}
		| \omega q_* [(\omega^{-1} q_*^{-1})_x]^2| \leq C \omega_{\eta, 0} \rho_{0, -3}
	\end{align}
	for a fixed $\eta > 0$ sufficiently small. In particular 
	\begin{align*}
		\| \omega q_* [ (\omega^{-1} q_*^{-1})_x ]^2 \psi(s)^2 \|_{L^1_{0, 1}} \leq C \| \omega_{\eta/2, 0} \rho_{0, -2} \|_{L^1} \| \psi(s) \|_{L^\infty_{-1, 0}}^2 \leq \frac{C}{(1+s)^2} \Theta(s)^2. 
	\end{align*}
	Hence we obtain 
	\begin{multline*}
		(1+t)^{3/2} \left\| \int_0^{t-1}  e^{\mcl_p (t-s)} \left[ \omega q_* [(\omega^{-1} q_*^{-1})_x ]^2 \psi(s)^2 \right] \, ds \right\|_{W^{1, \infty}_{0, -1}} \\ \leq C \Theta(t)^2 (1+t)^{3/2} \int_0^{t-1} \frac{1}{(t-s)^{3/2}} \frac{1}{(1+s)^2} \, ds 
		\leq C \Theta(t)^2
	\end{multline*}
	by Lemma \ref{l: nonlinearity integral estimate}, as desired. 
\end{proof}

The estimates on $\mathcal{I}_{p, 4} (t)$ are strictly easier than those on $\mathcal{I}_{p, i} (t), i = 2,3$, since the only difference is that the factor of $\omega q_*$ is replaced by a factor of $p(s)$, which has the same spatial localization but extra temporal decay: that is, we have $\| \omega q_* \|_{L^\infty_{0, -1}} \leq C$ but $\| p(s) \|_{W^{1,\infty}_{0, -1}} \leq C (1+s)^{-3/2} \Theta(s)$. We thereby obtain the following lemma. 

\begin{lemma}[$W^{1, \infty}_{0, -1}$ estimates on $\mathcal{I}_{p, 4} (t)$]\label{l: Ip4 W inf 0 -1}
	There exists a constant $C > 0$ such that for all $t \in (0, T_*)$, we have 
	\begin{align*}
		(1+t)^{3/2} \| \mathcal{I}_{p, 4} (t) \|_{W^{1, \infty}_{0, -1}} \leq C \Theta(t)^2
	\end{align*}	
	provided $\| \omega^{-1} q_*^{-1} p(s) \|_{L^\infty} \leq \frac{1}{2}$ for all $s \in (0, t)$.
\end{lemma}

\begin{prop}[$W^{1, \infty}_{0, -1}$ estimates on $p(t)$]\label{p: W1 inf 0 -1 estimates on p}
	There exist positive constants $C_1$ and $C_2$ such that 
	\begin{align}
		(1+t)^{3/2} \| p(t) \|_{W^{1,\infty}_{0,-1}} \leq C_1 \left( \| p_0 \|_{L^1_{0,1}} + \| p_0 \|_{W^{1,\infty}} \right) + C_2 \Theta(t)^2
	\end{align}
	provided $\| \omega^{-1} q_*^{-1} p(s) \|_{L^\infty} \leq \frac{1}{2}$ for all $s \in (0, t)$.
\end{prop}
\begin{proof}
	Having already handled all the nonlinear terms in Lemmas \ref{l: Ip1 L inf 0 -1} through \ref{l: Ip4 W inf 0 -1}, it only remains to estimate the term $e^{\mcl_p t} p_0$ in the variation of constants formula \eqref{e: voc p}. For $0 < t <1$, we have by Lemma \ref{l: small time estimates}
	\begin{align*}
		(1+t)^{3/2} \| e^{\mcl_p t} p_0 \|_{W^{1,\infty}_{0,-1}} \leq C (1+t)^{3/2} \| p_0 \|_{W^{1,\infty}} \leq C \| p_0 \|_{W^{1, \infty}}. 
	\end{align*}
	For $t > 1$, we instead use Proposition \ref{p: linear p t-3/2 estimate} to estimate 
	\begin{align*}
		(1+t)^{3/2} \| e^{\mcl_p t} p_0 \|_{W^{1,\infty}_{0,-1}} \leq C \frac{(1+t)^{3/2}}{t^{3/2}} \| p_0 \|_{L^1_{0,1}} \leq C \| p_0 \|_{L^1_{0, 1}},
	\end{align*}
	so that for all $t > 0$ we have 
	\begin{align*}
		(1+t)^{3/2} \| p(t) \|_{W^{1,\infty}_{0,-1}} \leq C_1 \left( \| p_0\|_{L^1_{0,1}} + \| p_0 \|_{W^{1,\infty}} \right),
	\end{align*}
	as desired. 
\end{proof}

We now establish control of $p(t)$ in $W^{1,1}$, which is in turn used in estimating several terms in the nonlinearity. 

\begin{prop}[$W^{1,1}$ estimates on $p(t)$]\label{p: nonlin W11 estimates on p}
	There exist positive constants $C_1$ and $C_2$ such that for all $t \in (0, T_*)$, we have 
	\begin{align}
		(1+t)^{1/2} \| p(t) \|_{W^{1,1}} \leq C_1 \left( \| p_0 \|_{L^1_{0,1}} + \|\partial_x p_0 \|_{L^1} \right)  + C_2 \Theta(t)^2
	\end{align}
	provided $\| \omega^{-1} q_*^{-1} p(s) \|_{L^\infty} \leq \frac{1}{2}$ for all $s \in (0, t)$.
\end{prop}
\begin{proof}
	The proof is exactly the same as that of Proposition \ref{p: W1 inf 0 -1 estimates on p}, but with Proposition \ref{p: p L11 to L1 estimate} replacing the linear estimate Proposition \ref{p: linear p t-3/2 estimate}. 
\end{proof}

\subsection{Estimates on $\psi(t)$}

We first prove estimates on $\| \psi(t) \|_{L^\infty_{-1,0}}$, since in light of Proposition \ref{p: psi L1 01 Linf-10 estimate}, these require measuring the nonlinearities in $L^1_{0,1}$, and so these estimates are strictly harder than estimates on $\| \psi(t) \|_{L^\infty}$, which only require measuring the nonlinearities in the weaker $L^1$ norm. 

\begin{lemma}[$L^\infty_{0, -1}$ estimates on $\mathcal{I}_{\psi, 1} (t)$]
	There exists a constant $C > 0$ such that for all $t \in (0, T_*)$, we have 
	\begin{align}
		(1+t) \| \mathcal{I}_{\psi, 1} (t) \|_{L^\infty_{-1, 0}} \leq C \Theta(t)^2
	\end{align}
	provided $\| \omega^{-1} q_*^{-1} p(s) \|_{L^\infty} \leq \frac{1}{2}$ for all $s \in (0, t)$. 
\end{lemma}
\begin{proof}
	We again split the integral in $\mathcal{I}_{\psi, 1} (t)$ as 
	\begin{align*}
		\mathcal{I}_{\psi, 1} (t) = \left( \int_0^{t-1} + \int_{t-1}^t \right) e^{\mcl_\psi (t-s)} \left[ 2 \frac{q_*'}{q_*} \left( \frac{1}{1+ \omega^{-1} q_*^{-1} p} - 1 \right) \psi_x (s) \right] \, ds.
	\end{align*}
	Throughout this section, we will focus on estimating the integral from $0$ to $t-1$. The estimates on the other integral are similar, but we replace Proposition \ref{p: psi L1 01 Linf-10 estimate} with Lemma \ref{l: small time estimates} to guarantee integrability near $t = s$.  Of course if $t < 1$, then we only write one integral from $0$ to $t$, and use the small time estimates of Lemma \ref{l: small time estimates} in this single integral. Hence for the remainder of this section we will assume $t > 1$. 
	Using Proposition \ref{p: psi L1 01 Linf-10 estimate}, we estimate 
	\begin{multline*}
		(1+t) \left\| \int_0^{t-1} e^{\mcl_\psi (t-s)} \left[ 2 \frac{q_*'}{q_*} \left( \frac{1}{1+ \omega^{-1} q_*^{-1} p} - 1 \right) \psi_x (s) \right] \, ds \right\|_{L^\infty_{-1,0}} \\ \leq C (1+t) \int_0^{t-1} \frac{1}{(t-s)} \left\| \frac{q_*'}{q_*} \left( \frac{1}{1 + \omega^{-1} q_*^{-1} p} - 1 \right) \psi_x(s) \right\|_{L^1_{0,1}} \, ds
	\end{multline*}
	Taylor expanding the nonlinearity and using that $|\omega^{-1} q_*^{-1}| \leq C \rho_{0, -1}$ and $q_*'/q_*$ is bounded, we have 
	\begin{align*}
		\left\| \frac{q_*'}{q_*} \left( \frac{1}{1 + \omega^{-1} q_*^{-1} p} - 1 \right) \psi_x(s) \right\|_{L^1_{0,1}}
		&\leq C \left\| \omega^{-1} q_*^{-1} p(s) \psi_x (s) \right\|_{L^1_{0,1}} \\
		&\leq C \| p(s) \|_{L^1} \| \psi_x(s) \|_{L^\infty}\\
		&\leq C \Theta(s)^2 \frac{1}{(1+s)^{1/2}} \frac{1}{(1+s)}. 
	\end{align*}
	Estimating the integral from $t-1$ to $t$ similarly but using the small time estimates from Lemma \ref{l: small time estimates}, we obtain 
	\begin{align*}
		(1+t)\| \mathcal{I}_{\psi, 1} (t) \|_{L^\infty_{0, -1}} &\leq C (1+t)  \Theta(t)^2 \left(\int_0^{t-1}  \frac{1}{t-s} \frac{1}{(1+s)^{3/2}} \, ds + \int_{t-1}^t \frac{1}{(t-s)^{1/2}} \frac{1}{(1+s)^{3/2}} \, ds \right) \\
		&\leq C  \Theta(t)^2
	\end{align*}
	by Lemma \ref{l: nonlinearity integral estimate} and an analogous argument for the second integral, provided $\| \omega^{-1} q_*^{-1} p(s) \|_{L^\infty} \leq \frac{1}{2}$ for all $s \in (0, t)$, as desired. 
\end{proof}

\begin{lemma}[$L^\infty_{0, -1}$ estimates on $\mathcal{I}_{\psi, 2} (t)$] 
	There exists a constant $C > 0$ such that for all $t \in (0, T_*)$, we have 
	\begin{align}
		(1+t) \| \mathcal{I}_{\psi, 2} (t) \|_{L^\infty_{-1, 0}} \leq C \Theta(t)^2
	\end{align}
	provided $\| \omega^{-1} q_*^{-1} p(s) \|_{L^\infty} \leq \frac{1}{2}$ for all $s \in (0, t)$. 
\end{lemma}
\begin{proof}
	We again focus on the integral from $0$ to $t-1$ in the definition of $\mathcal{I}_{\psi, 2}(t)$. By Proposition \ref{p: psi L1 01 Linf-10 estimate}, we have 
	\begin{multline*}
		(1+t) \left\| \int_0^{t-1} e^{\mcl_\psi (t-s)} \left[ 2 \omega q_*' \left( \frac{1}{1+ \omega^{-1} q_*^{-1} p} - 1 \right) (\omega^{-1} q_*^{-1})_x \psi(s) \right] \, ds \right\|_{L^\infty_{-1, 0}} \\ \leq C (1+t) \int_0^{t-1} \frac{1}{(t-s)} \left\| \omega q_*' \left( \frac{1}{1+\omega^{-1} q_*^{-1} p} -1 \right) (\omega^{-1} q_*^{-1})_x \psi(s) \right\|_{L^1_{0, 1}} \, ds. 
	\end{multline*}
	Taylor expanding the nonlinearity and exploiting the localization of $(\omega^{-1} q_*^{-1})_x$ implied by \eqref{e: front asymptotics}, we obtain 
	\begin{align*}
		\left\| \omega q_*' \left( \frac{1}{1+\omega^{-1} q_*^{-1} p} -1 \right) (\omega^{-1} q_*^{-1})_x \psi(s) \right\|_{L^1_{0, 1}} &\leq C  \| p(s) \|_{L^1} \| \psi(s) \|_{L^\infty_{-1, 0}} \\
		&\leq C \Theta(s)^2 \frac{1}{(1+s)^{1/2}} \frac{1}{(1+s)}. 
	\end{align*}
	Estimating the integral from $t-1$ to $t$ similarly, we obtain  
	\begin{align*}
		(1+t) \| \mathcal{I}_{\psi, 2} (t) \|_{L^\infty_{-1, 0}} &\leq C \Theta(t)^2 \left( \int_0^{t-1} \frac{1}{(t-s)} \frac{1}{(1+s)^{3/2}} \, ds + \int_{t-1}^t \frac{1}{(t-s)^{1/2}} \frac{1}{(1+s)^{3/2}} \, ds \right) \\ &\leq C  \Theta(t)^2
	\end{align*}
	by Lemma \ref{l: nonlinearity integral estimate}, provided $\| \omega^{-1} q_*^{-1} p(s) \|_{L^\infty} \leq \frac{1}{2}$ for all $s \in (0, t)$, as desired. 
\end{proof}

\begin{lemma}[$L^\infty_{0, -1}$ estimates on $\mathcal{I}_{\psi, 3} (t)$]
	There exists a constant $C > 0$ such that for all $t \in (0, T_*)$, we have
	\begin{align}
		(1+t) \| \mathcal{I}_{\psi, 3} (t) \|_{L^\infty_{-1, 0}} \leq C  \Theta(t)^2
	\end{align}
	provided $\| \omega^{-1} q_*^{-1} p(s) \|_{L^\infty} \leq \frac{1}{2}$ for all $s \in (0,t)$. 
\end{lemma}
\begin{proof}
	Again, we focus on the integral from $0$ to $t-1$ in the definition of $\mathcal{I}_{\psi, 3} (t)$. By Proposition \ref{p: psi L1 01 Linf-10 estimate}, we have 
	\begin{multline*}
		(1+t) \left\| \int_0^{t-1} e^{\mcl_\psi(t-s)} \left[ 2 \omega (\omega^{-1} p)_x \left( \frac{1}{1+ \omega^{-1} q_*^{-1}} \right) (\omega^{-1} q_*^{-1} \psi(s))_x \right] \, ds \right\|_{L^\infty_{-1, 0}} \\ \leq C (1+t) \int_0^{t-1} \frac{1}{(t-s)} \left\| \omega (\omega^{-1} p )_x \left( \frac{1}{1+ \omega^{-1} q_*^{-1} p} \right) (\omega^{-1} q_*^{-1} \psi(s))_x \right\|_{L^1_{0, 1}} \, ds. 
	\end{multline*}
	We estimate the nonlinearity as
	\begin{align*}
		\left\| \omega (\omega^{-1} p )_x \left( \frac{1}{1+ \omega^{-1} q_*^{-1} p} \right) (\omega^{-1} q_*^{-1} \psi(s))_x \right\|_{L^1_{0, 1}} &\leq C  \|p(s) \|_{W^{1,1}} \left( \| \psi_x(s) \|_{L^\infty} + \| \psi(s) \|_{L^\infty_{-1, 0}} \right) \\
		&\leq C  \Theta(s)^2 \frac{1}{(1+s)^{1/2}} \left( \frac{1}{1+s} + \frac{1}{1+s} \right) \\
		&\leq C  \frac{1}{(1+s)^{3/2}} \Theta(s)^2,
	\end{align*}
	provided $\| \omega^{-1} q_*^{-1} p(s) \|_{L^\infty} \leq \frac{1}{2}$ for all $s \in (0, t)$. Hence we obtain 
	\begin{align*}
		(1+t) \| \mathcal{I}_{\psi, 3} (t) \|_{L^\infty_{-1, 0}} &\leq C  \Theta(t)^2 (1+t) \left( \int_0^{t-1} \frac{1}{t-s} \frac{1}{(1+s)^{3/2}} \, ds + \int_{t-1}^t \frac{1}{(t-s)^{1/2}} \frac{1}{(1+s)^{3/2}} \, ds \right) \\ 
		&\leq C  \Theta(t)^2
	\end{align*}
	by Lemma \ref{l: nonlinearity integral estimate}, as desired. 
\end{proof}

Combining the preceding three lemmas with the small time estimates from Section \ref{s: linear estimates} as in the proof of Proposition \ref{p: W1 inf 0 -1 estimates on p}, we obtain the following control of $\| \psi(t) \|_{L^\infty_{-1, 0}}$. 

\begin{prop}[$L^\infty_{0, -1}$ estimates on $\psi(t)$]\label{p: nonlin psi L inf 0 -1 estimates}
	There exist positive constants $C_1$ and $C_2$ such that for all $t \in (0, T_*)$, we have 
	\begin{align}
		(1+t) \| \psi(t) \|_{L^\infty_{-1, 0}} \leq C_1 \left( \| \psi_0 \|_{L^1_{0,1}} + \|\psi_0\|_{L^\infty} \right)+ C_2 \Theta(t)^2
	\end{align}
	provided $\| \omega^{-1} q_*^{-1} p(s) \|_{L^\infty} \leq \frac{1}{2}$ for all $s \in (0, t)$.
\end{prop}
\begin{proof}
		Having estimated the nonlinear terms in the preceding three lemmas, it only remains to estimate the term $e^{\mcl_\psi t} \psi_0$ in the variation of constants formula \eqref{e: voc psi}. For $0 < t < 1$, we have by Lemma \ref{l: small time estimates}
	\begin{align*}
		(1+t) \| e^{\mcl_\psi t} \psi_0 \|_{L^\infty_{-1,0}} \leq C (1+t) \| \psi_0 \|_{L^\infty} \leq C \| \psi_0 \|_{L^\infty}. 
	\end{align*}
	For $t > 1$, we instead use Proposition \ref{p: psi L1 01 Linf-10 estimate} to estimate 
	\begin{align*}
		(1+t) \| e^{\mcl_\psi t} \psi_0 \|_{L^\infty_{-1,0}} \leq C \frac{(1+t)}{t} \| \psi_0\|_{L^1_{0,1}} \leq C \| \psi_0 \|_{L^1_{0,1}},
	\end{align*}
	which completes the proof of the proposition. 
\end{proof}

By similar arguments, we obtain the following control of $\psi(t)$ in $L^\infty$. 

\begin{prop}[$L^\infty$ estimates on $\mathcal{\psi}(t)$]\label{p: psi L inf estimates}
	There exist positive constants $C_1$ and $C_2$ such that for all $t \in (0, T_*)$, we have 
	\begin{align}
		(1+t)^{1/2} \| \psi(t)\|_{L^\infty} \leq C_1 \left( \| \psi_0 \|_{L^1} + \| \psi_0 \|_{L^\infty} \right)+ C_2 \Theta(t)^2
	\end{align}
	provided $\| \omega^{-1} q_*^{-1} p(s) \|_{L^\infty} \leq \frac{1}{2}$ for all $s \in (0, t)$. 
\end{prop}
\begin{proof}
	The proof is exactly the same as that of Proposition \ref{p: nonlin psi L inf 0 -1 estimates} but with the estimate $\|e^{\mcl_\psi t}\|_{L^1 \to L^\infty} \leq Ct^{-1/2}$ of Proposition \ref{p: psi L1 Linf estimate} replacing Proposition \ref{p: psi L1 01 Linf-10 estimate}. The control of the nonlinearities in the stronger norm $L^1_{0,1}$ obtained in the proof of Proposition \ref{p: nonlin psi L inf 0 -1 estimates} is sufficient to close the argument. 
\end{proof}

It only remains to estimate $\psi_x (t)$ in the spaces encoded in the definition of $\Theta(t)$, \eqref{e: theta def}. To do this, we differentiate the variation of constants formula \eqref{e: voc psi}, obtaining 
\begin{align}
	\psi_x (t) = \partial_x (e^{\mcl_\psi t} \psi_0 ) + \partial_x \mathcal{I}_{\psi,1} (t) + \partial_x \mathcal{I}_{\psi, 2} (t) + \partial_x \mathcal{I}_{\psi, 3} (t). \label{e: voc psi x}
\end{align}
Notice that the linear estimates on $\partial_x e^{\mcl_\psi t}$ all measure the initial data in the $L^1$ norm. Hence the nonlinear estimates on derivatives are again strictly easier than those obtained in the proof of Proposition \ref{p: nonlin psi L inf 0 -1 estimates}, which requires measuring the nonlinearities in the stronger norm $L^1_{0,1}$, and so we readily obtain the following nonlinear estimates on derivatives. 

\begin{prop}[Estimates on $\psi_x(t)$]\label{p: nonlinear psi derivative estimates}
	There exist positive constants $C_1$ and $C_2$ such that for all $t \in (0, T_*)$, we have 
	\begin{align}
		(1+t) \| \psi_x (t) \|_{L^\infty} + t^{1/2} \| \psi_x (t) \|_{L^1} \leq C_1 \left( \| \psi_0 \|_{L^1} +\|\psi_0 \|_{W^{1, \infty}} \right) + C_2 \Theta(t)^2
	\end{align}
	provided $\| \omega^{-1} q_*^{-1} p(s) \|_{L^\infty} \leq \frac{1}{2}$ for all $s \in (0, t)$. 
\end{prop}

\begin{lemma}[Small time control of $\psi_{xx}(t)$]\label{l: nonlin small time psi xx control}
	There exist positive constants $C_1$ and $C_2$ such that for all $0 < t < \min(1, T_*)$, we have 
	\begin{align}
		t^{1/2} \| \psi_{xx}(t) \|_{L^\infty} \leq C \left( \| \psi_0 \|_{W^{1,\infty}} \right) + C_2 \Theta(t)^2,
	\end{align}
\end{lemma}
provided $\| \omega^{-1} q_*^{-1} p(s) \|_{L^\infty} \leq \frac{1}{2}$ for all $s \in (0,t)$. 
\begin{proof}
	We differentiate the variation of constants formula \eqref{e: voc psi} twice, use the estimate 
	\begin{align*}
		\|\partial_{xx} (e^{\mcl_\psi t} \psi_0 )\|_{L^\infty} \leq C t^{-1/2} \|\psi_0 \|_{W^{1,\infty}}
	\end{align*} 
	from Lemma \ref{l: small time estimates}, and estimate the nonlinear terms in a similar manner to the proof of Proposition \ref{p: psi L inf estimates} for small times. 
\end{proof}

The desired control of $\Theta(t)$, Proposition \ref{p: theta control}, follows readily from the control of the individual terms from Propositions \ref{p: W1 inf 0 -1 estimates on p}, \ref{p: nonlin psi L inf 0 -1 estimates}, \ref{p: psi L inf estimates},  and \ref{p: nonlinear psi derivative estimates} together with Lemma \ref{l: nonlin small time psi xx control}. 

\begin{proof}[Proof of Theorem \ref{t: main}]
	Define 
	\begin{align*}
		\Omega_0 := \| p_0 \|_{L^1_{0, 1}} + \| p_0 \|_{W^{1, \infty} \cap W^{1,1}} + \| \psi_0 \|_{L^1_{0, 1}} + \| \psi_0 \|_{W^{1, \infty}}. 
	\end{align*}
	By Proposition \ref{p: theta control}, there exist positive constants $C_1$ and $C_2$ such that 
	\begin{align}
		\Theta(t) \leq C_1 \Omega_0 + C_2 \Theta(t)^2 \label{e: thm 1 proof theta control}
	\end{align}
	for all $t \in (0, T_*)$ provided $\| \omega^{-1} q_*^{-1} p(s) \|_{L^\infty} \leq \frac{1}{2}$ for all $s \in (0, t)$. It follows from the local well-posedness theory that there exists a constant $C_0 > 0$ such that $\Theta(t) \leq C_0 \Omega_0$ for $t$ sufficiently small. Set $\tilde{C}_1 = \max ( C_1, C_0)$, and suppose $\Omega_0$ is sufficiently small so that 
	\begin{align}
		2 \tilde{C}_1 \Omega_0 < \frac{1}{2} \quad \text{and} \quad 4 \tilde{C}_1 C_2 \Omega_0 < 1. 
	\end{align}
	In particular, $\Theta(t) \leq C_0 \Omega_0 < \frac{1}{2}$ for $t$ sufficiently small, so that \eqref{e: thm 1 proof theta control} holds for $t$ sufficiently small. We show that 
	\begin{align}
		\Theta(t) \leq 2 \tilde{C}_1 \Omega_0 < \frac{1}{2} \label{e: theorem 1 proof theta control}
	\end{align} 
	for all $t \in (0, T_*)$. By construction, $t \mapsto \Theta(t)$ is continuous on $(0, T_*)$. Hence if \eqref{e: theorem 1 proof theta control} does not hold, then there is some time $t_1$ at which $\Theta(t_1) = 2 \tilde{C}_1 \Omega_0$. Considering \eqref{e: thm 1 proof theta control} at time $t_1$, we obtain 
	\begin{align*}
		2 \tilde{C}_1 \Omega_0 \leq \tilde{C}_1 \Omega_0 + 4 (2 \tilde{C}_1 \Omega_0) \tilde{C}_1^2 C_2 \Omega_0^2 \leq \tilde{C}_1 \Omega_0 (1 + 4 \tilde{C}_1 C_2 \Theta_0),
	\end{align*}
	using that in particular $2 \tilde{C}_1 \Omega_0 < 1$. Since we also have $4 \tilde{C}_1 C_2 \Omega_0 < 1$ by assumption, we conclude 
	\begin{align*}
		2 < 1 + 4 \tilde{C}_1 C_2 K (B) \Omega_0 < 2,
	\end{align*}
	a contradiction. Hence $\Theta(t) \leq 2 \tilde{C}_1 \Omega_0$ for all $t \in (0, T_*)$, which implies by the local well-posedness theory that $T_* = \infty$. This global control of $\Theta(t)$ implies in particular 
	\begin{align*}
		\| p(t) \|_{W^{1,\infty}_{0,-1}} \leq \frac{C}{(1+t)^{3/2}} \Omega_0
	\end{align*}
	and 
	\begin{align*}
		\| \psi(t)\|_{W^{1,\infty}} \leq \frac{C}{(1+t)^{1/2}} \Omega_0. 
	\end{align*}
	Reverting to $(r, \phi)$ coordinates and using the fact that there exist constants $c, C > 0$ so that $c \rho_{0,-1} \leq | \omega^{-1} q_*^{-1}| \leq C \rho_{0, -1}$, we find that these estimates are equivalent to those stated in Theorem \ref{t: main}, and the smallness of $\Omega_0$ translates to the smallness condition on the initial data in the statement there. We note that from the control of $\Theta(t)$ we further obtain detailed estimates on $p$ and $\psi$ in stronger norms, as well as decay estimates on derivatives. 
\end{proof}

%

\begin{thebibliography}{10}
	
	\bibitem{aronson}
	D.~Aronson and H.~Weinberger.
	\newblock Multidimensional nonlinear diffusion arising in population genetics.
	\newblock {\em Adv. Math.}, 30(1):33--76, 1978.
	
	\bibitem{AveryScheelSelection}
	M.~Avery and A.~Scheel.
	\newblock Universal selection of pulled fronts.
	\newblock {\em Preprint}.
	\newblock arXiv:2012.06443 [math.AP].
	
	\bibitem{AveryScheel}
	M.~Avery and A.~Scheel.
	\newblock Asymptotic stability of critical pulled fronts via resolvent
	expansions near the essential spectrum.
	\newblock {\em SIAM J. Math. Anal.}, 53(2):2206--2242, 2021.
	
	\bibitem{BeckNguyenSandstedeZumbrunCGL}
	M.~Beck, T.~Nguyen, B.~Sandstede, and K.~Zumbrun.
	\newblock Nonlinear stability of source defects in the complex
	{G}inzburg-{L}andau equation.
	\newblock {\em Nonlinearity}, 27:739--786, 2014.
	
	\bibitem{Bramson1}
	M.~Bramson.
	\newblock Maximal displacement of branching {B}rownian motion.
	\newblock {\em Comm. Pure Appl. Math.}, 31(5):531--581, 1978.
	
	\bibitem{Bramson2}
	M.~Bramson.
	\newblock {\em Convergence of solutions of the {K}olmogorov equation to
		traveling waves}.
	\newblock Mem. Amer. Math. Soc. American Mathematical Society, 1983.
	
	\bibitem{Bricmont}
	J.~Bricmont and A.~Kupiainen.
	\newblock Stability of moving fronts in the {G}inzburg-{L}andau equation.
	\newblock {\em Comm. Math. Phys}, 159:287--318, 1994.
	
	\bibitem{EbertvanSaarloos}
	U.~Ebert and W.~van Saarloos.
	\newblock Front propagation into unstable states: universal algebraic
	convergence towards uniformly translating pulled fronts.
	\newblock {\em Phys. D}, 146:1--99, 2000.
	
	\bibitem{SchneiderEckmann}
	J.-P. Eckmann and G.~Schneider.
	\newblock Non-linear stability of modulated fronts for the {S}wift-{H}ohenberg
	equation.
	\newblock {\em Comm. Math. Phys.}, 225:361--397, 2002.
	
	\bibitem{EckmannWayne}
	J.-P. Eckmann and C.~E. Wayne.
	\newblock The nonlinear stability of front solutions for parabolic partial
	differential equations.
	\newblock {\em Comm. Math. Phys}, 161(2):323--334, 1994.
	
	\bibitem{FayeHolzer}
	G.~Faye and M.~Holzer.
	\newblock Asymptotic stability of the critical {F}isher--{KPP} front using
	pointwise estimates.
	\newblock {\em Z. Angew. Math. Phys.}, 70(1):13, 2018.
	
	\bibitem{FiedlerScheel}
	B.~Fiedler and A.~Scheel.
	\newblock Spatio-temporal dynamics of reaction-diffusion patterns.
	\newblock In {\em Trends in Nonlinear Analysis}, pages 23--152, Berlin,
	Heidelberg, 2003. Springer, Berlin Heidelberg.
	
	\bibitem{Gallay}
	T.~Gallay.
	\newblock Local stability of critical fronts in nonlinear parabolic partial
	differential equations.
	\newblock {\em Nonlinearity}, 7(3):741--764, 1994.
	
	\bibitem{Comparison1}
	F.~Hamel, J.~Nolen, J.-M. Roquejoffre, and L.~Ryzhik.
	\newblock A short proof of the logarithmic {B}ramson correction in
	{F}isher-{KPP} equations.
	\newblock {\em Netw. Heterog. Media}, 8(1):275--289, 2013.
	
	\bibitem{Henry}
	D.~Henry.
	\newblock {\em Geometric theory of semilinear parabolic equations}.
	\newblock Lecture Notes in Math. Springer-Verlag, Berlin Heidelberg, 1981.
	
	\bibitem{Hilder}
	B.~Hilder.
	\newblock Nonlinear stability of fast invading fronts in a {G}inzburg-{L}andau
	equation with an additional conservation law.
	\newblock {\em Nonlinearity}, 34(8):5538--5575, 2021.
	
	\bibitem{HolzerScheelPointwiseGrowth}
	M.~Holzer and A.~Scheel.
	\newblock Criteria for pointwise growth and their role in invasion processes.
	\newblock {\em J. Nonlinear Sci.}, 24(1):661--709, 2014.
	
	\bibitem{ZumbrunHoward}
	P.~Howard and K.~Zumbrun.
	\newblock Pointwise semigroup methods and stability of viscous shock waves.
	\newblock {\em Indiana Univ. Math. J.}, 47(3):741--871, 1998.
	
	\bibitem{Kapitula2}
	T.~Kapitula.
	\newblock Stability of weak shocks in {$\omega$--$\lambda$} systems.
	\newblock {\em Indiana U. Math. J.}, 40(4):1193--1219, 1991.
	
	\bibitem{Kapitula1}
	T.~Kapitula.
	\newblock On the stability of traveling waves in weighted {$L^\infty$} spaces.
	\newblock {\em J. Differential Equations}, 112(1):179--215, 1994.
	
	\bibitem{KapitulaPromislow}
	T.~Kapitula and K.~Promislow.
	\newblock {\em Spectral and dynamical stability of nonlinear waves}.
	\newblock Appl. Math, Sci. Springer, New York, 2013.
	
	\bibitem{Lunardi}
	A.~Lunardi.
	\newblock {\em Analytic semigroups and optimal regularity in parabolic
		problems}.
	\newblock Progr. Nonlinear Differential Equations Appl. Birkhauser, Basel,
	1995.
	
	\bibitem{Comparison2}
	J.~Nolen, J.-M. Roquejoffre, and L.~Ryzhik.
	\newblock Convergence to a single wave in the {F}isher-{KPP} equation.
	\newblock {\em Chin. Ann. Math. Ser. B}, 38(2):629--646, 2017.
	
	\bibitem{Comparison3}
	J.~Nolen, J.-M. Roquejoffre, and L.~Ryzhik.
	\newblock Refined long-time asymptotics for {F}isher–{KPP} fronts.
	\newblock {\em Commun. Contemp. Math.}, 21(07):1850072, 2019.
	
	\bibitem{Palmer1}
	K.~Palmer.
	\newblock Exponential dichotomies and transversal homoclinic points.
	\newblock {\em J. Differential Equations}, 55:225--256, 1984.
	
	\bibitem{Palmer2}
	K.~Palmer.
	\newblock Exponential dichotomies and {F}redholm operators.
	\newblock {\em Proc. Amer. Math. Soc.}, 104:149--156, 1988.
	
	\bibitem{Sattinger}
	D.~Sattinger.
	\newblock On the stability of waves of nonlinear parabolic systems.
	\newblock {\em Adv. Math.}, 22(3):312--355, 1976.
	
	\bibitem{vanSaarloosReview}
	W.~van Saarloos.
	\newblock Front propagation into unstable states.
	\newblock {\em Phys. Rep.}, 386:29--222, 2003.
	
\end{thebibliography}

\end{document}